\documentclass[a4paper]{article}
\usepackage{amsmath,amssymb,amsthm,amsfonts,graphicx,fullpage}%,bookman}
\usepackage{epsfig}
\usepackage{color}
\usepackage[applemac]{inputenc}
\usepackage{url}
\usepackage{enumerate}
\usepackage[pdftex]{hyperref}
\usepackage[nottoc]{tocbibind}
\usepackage{soul} % tachar palabras

\usepackage{boites}

\def\Aset{\mathbb{A}}

\def\Nset{\mathbb{N}}

\def\Rset{\mathbb{R}}

\def\diag{{\mathrm {\tt diag}}}

\def\bfo{{\mathbf 1}}

\def\tmu{\tilde{\mu}}
\def\tm{\tilde{m}}

\def\meas{{\mathrm {meas}}}

\def\hL{\hat{L}}
\def\hv{\hat{v}}
\def\hnu{\hat{\nu}}

\definecolor{bbleu}{rgb}{0.35, 0.01, 0.85}
\definecolor{vert}{rgb}{0.04, 0.65, 0.4}
\definecolor{orange}{rgb}{1,0.5,0}

\newcommand{\comment}[1]{}
\renewcommand{\t}{^{\mbox{\tiny\sf T}}}

\providecommand{\keywords}[1]{\textbf{{Keywords:}}#1}

\newtheorem{theo}{Theorem}
\newtheorem*{theo*}{Theorem}
\newtheorem{defi}[theo]{Definition}

\newtheorem{lemm}[theo]{Lemma}

\newtheorem{assumption}{Assumption}
\newtheorem*{prop*}{Proposition}

\theoremstyle{definition}

\newtheorem*{rema*}{Remark}
\newtheorem*{remas*}{Remarks}

\title{ A fast-slow model for adaptive resistance evolution}

\author{Pastor E. P\'erez-Estigarribia\thanks{ Polytechnic School, National University of Asunci\'on, P.O.\ Box 2111 SL, San Lorenzo, Paraguay, email: \href{mailto:pestigarribia@est.pol.una.py}{pestigarribia@est.pol.una.py}} \and Pierre-Alexandre Bliman\thanks{Sorbonne Universit\'e, Universit\'e Paris-Diderot SPC, Inria, CNRS, Laboratoire Jacques-Louis Lions, \'equipe Mamba, 75005 Paris, France. {\bf Corresponding author}, email: \href{mailto:pierre-alexandre.bliman@inria.fr}{pierre-alexandre.bliman@inria.fr}} \and Christian E.\ Schaerer\thanks{ Polytechnic School, National University of Asunci\'on, P.O.\ Box 2111 SL, San Lorenzo, Paraguay, email: \href{mailto:cschaer@pol.una.py}{cschaer@pol.una.py}}}

\begin{document}
	
	\maketitle
	
	\tableofcontents
	
\begin{abstract}
Resistance to insecticide is considered nowadays one of the major threats to insect control, as its occurrence reduces drastically the efficiency of chemical control campaigns, and may also perturb the application of other control methods, like biological and genetic control. In order to account for the emergence and spread of such phenomenon as an effect of exposition to larvicide and/or adulticide, we develop in this paper a general time-continuous population model with two life phases, subsequently simplified through slow manifold theory. The derived models present density-dependent recruitment and mortality rates in a non-conventional way.
We show that in absence of selection, they evolve in compliance with Hardy-Weinberg law; while in presence of selection and in the dominant or codominant cases, convergence to the fittest genotype occurs.
The proposed mathematical models should allow for the study of several issues of importance related to the use of insecticides and other adaptive phenomena.
\end{abstract}


\section{Introduction}

\subsection{Insect pests and insecticide use}

A tiny insect is one of the deadliest animals for human: the mosquitoes.
Nearly 700 million people contract diseases transmitted by mosquito every year \cite{caraballo2014emergency}.
Just for the transmission of malaria, almost a million people die every year \cite{201827}, and 3.4 billion people are at risk worldwide  \cite{world2014global}.
Amongst the 150 arboviruses that cause diseases in humans, about 20 that are transmitted by mosquitoes are of primary medical importance \cite{roehrig2009arboviruses}.
In particular dengue, transmitted by the primary vector {\em Aedes aegypti} and whose human and economic costs are staggering, is ranked as the most important mosquito-borne viral disease with epidemic potential in the world \cite{bhatt2013global,world2014global,shepard2016global}.
On the whole, vector-borne diseases generate a heavy public health problem and high economic cost in many world regions.

Besides, it has been suggested that insects destroy about 20\% of the annual production of crops worldwide \cite{sharma2017insect}. The agricultural damage caused by insects may be direct, as well as indirect, through the transmission of plant diseases \cite{oerke2006crop}.
Two factors contribute to the importance of insects as agricultural pests: their diversity --- two thirds of known species, around  600,000, are phytophagous ---, and the fact that practically all plant species are consumed by at least one species of phytophagous insect --- even some agricultural pests can hit many species of plants \cite{douglas2018strategies}.
The issues induced by agricultural pests are therefore a challenge for global food production. 

The most commonly used methods to suppress or reduce insect populations are based on chemical treatment.
Conventional control strategy of mosquito populations use larvicides and/or adulticides \cite{valle2015controle,ShawCatterucci}. However, this strategy is affected by the evolution of resistance \cite{schechtman2015costly, hemingway2000insecticide} which reduces the efficiency of the chemical control campaigns and their lifespan \cite{koella2009}. Nowadays, in addition to an alarming propagation of vectors, most of the species involved are showing resistance to many kinds of insecticides \cite{world2014global,vontas2012insecticide}.
Resistance not only reduces the efficiency of chemical control methods, but may also perturb the application of other control methods, like biological and genetic controls, through the undesired fitness advantage that it provides to the local mosquito species, see e.g.\ \cite{Azambuja-Garcia:2017aa,Garcia:2017aa}.

\subsection{Modelling of insecticide resistance evolution}

Resistance to insecticides is a man-made example of natural selection, and the factors that govern its origin and spread are of academic interest and of applied importance \cite{daborn2004genetics}.
The issue of insecticide resistance provides a contemporary natural model to study how new adaptations evolve by  selection: the selection agent is known (as an example a given insecticide), the evolution is recent and rapid (few years after the application), and the biological and genetic mechanisms are often known (many insect genes that code the targets for insecticides have been identified and cloned \cite{labbe2011evolution}).
In this regard, due to the cost, risk and logistical difficulties in the field study of resistance evolution, verisimilar mathematical models may help to improve management strategies of insecticides, apart from allowing to learn more about adaptive evolution in the context of complex life history.

Literature related to models of evolution by natural selection is quite extensive.
For an overview of the state of knowledge, the reader is referred to the following classical texts \cite{hofbauer1998evolutionary,ewens2011changes,schuster2011mathematical,schuster2011mathematics,burger2011some}.
Reproduction and natural selection involve inherently stochastic effects, and a typical modelling approach includes stochastic time-discrete processes \cite{huillet:hal-00526859}, which may however present 
considerable mathematical complexity \cite{burger2011some}.
In classical population genetics selection models, it is common to appeal to the law of large numbers and assume an infinite population size, to omit the effect of stochasticity and to focus the mathematical analysis on the evolution of the allelic and genotypic relative frequencies by selection.
This is the case of Fisher selection equation or evolutionary game theory (see \cite{schuster1983replicator,hofbauer1982game,hofbauer1998evolutionary}).
However, in control strategies one wishes to study, in addition to changes in allele relative frequencies, the efficiency in reducing or suppressing population.

For simplicity, it is common in population genetics, that the elementary mechanisms by which fitness is reduced are not specified \cite{barbosa2012importance, levick2017two}, in other words it is not clear if the measure of fitness is related to fertility, viability or both.
However, considering the diversity of life history in insects, realistic scenarios require a more specific treatment of the mechanisms that affect fitness. 
For example, the use of larvicides and/or adulticides for the control of mosquitoes affects viability differently in the pre-reproductive and reproductive phases. In addition, viability may decrease in a growing population due to the limitation of some resources by competition, and may depend on density, which in turn may affect %differently 
different classes of individuals in the population on a non-homogeneous way.
	
Different approaches have been proposed in the literature to model insecticide resistance, with several objectives.
In \cite{taylor1975insecticide} a time-discrete model is presented for three genotypes that considers a resistance allele with intermediate dominance in an autosomal locus.
A time-discrete model of insecticide resistance is developed in \cite{comins1977development} that includes migration, in order to understand variations in the time required for insects to develop resistance.
Likewise, \cite{taylor1979suppression} explored a model proposed in \cite{crow1970introduction} to quantify changes in gene frequency between zygotes in each generation.
The model in question is time-discrete, with intermediate dominance alleles in an autosomal locus, and assumes migration and a density-dependent growth of the population.
On the other hand, \cite{curtis1985theoretical} and \cite{mani1985evolution} used time-discrete models classical in populations genetics to show the essential characteristics of the selection dynamics for resistance by one or two insecticides on two autosomal locus.
More recently, this line of research was explored again \cite{levick2017two,south2018insecticide}.
Also, \cite{barbosa2012importance} developed genetic models to predict changes in the fitness and frequency of resistance alleles in synergistic scenarios. All these models consider Mendelian % heredity
inheritance, but do not take into account the existence of %different
several stages of life, and different selective pressure for each phase.

Furthermore, models have been proposed to develop new approaches, such as the idea of evolution-proof insecticides \cite{koella2009,read2009make,gourley2011slowing}. In particular, these models were developed to explore the control of malaria vector mosquito. Since the malaria parasite requires a mosquito with a long life expectancy, these authors explore strategies to slow down the evolution of resistance using combinations of larvicides and late-life-acting insecticides. In this regard a classical population genetics model is used in \cite{read2009make} to quantify the change in relative frequency of a resistant allele, assuming a constant population of adults. A population genetics model that considers the change in relative frequency of a resistance allele in different age classes is considered in \cite{koella2009}. The two works cited above differ from \cite{gourley2011slowing}, in which an aged-structured population of variable size is modelled by delay differential equations. However, this model does not possess inheritance mechanism for autosomal genes.
	
A non-traditional approach to resistance models, related to the one presented in the present paper, is given in \cite{langemann2013multi}.
Inspired by the study of herbicide resistant weed, this paper proposes a continuous-time deterministic model that combines the logistic growth in a population with the rearrangement of alleles in the genotypes given by inheritance.
In addition, the model applies to cases of multiple loci using tensor product, and extends to polyploid and other numbers of alleles.
However, it does not allow to model the several phases of an insect life, which present different selective pressures in their respective ecological niches.

Last, \cite{schechtman2015costly} proposed a compartmental model with age-structure to quantify the time required to reverse resistance in a dengue vector mosquito.
A model of dimension 15 (five life stages for three genotypes) is introduced, and numerical simulations are conducted, in order to evaluate the loss of resistance, assuming that in absence of insecticide the resistant genotype has lower fitness.

\subsection{The proposed modelling approach}
	
\comment{The mere formulation of a problem is far more essential than its solution, which may be merely a matter of mathematical or experimental skill. To raise new questions, new possibilities, to regard old problems from a new angle requires creative imagination and marks real advances in science. ---Einstein}

The issue studied in the present paper is the formulation of a model of the evolution of resistance to larvicides and/or adulticides on an autosomal gene for an insect population, taking into account the complexity of life history; as well as an analysis of this model.
	
Insects are diploid organisms with sexual reproduction and usually with several phases of life. In many cases the genes involved in insecticides resistance are autosomal and follow principles of Mendelian inheritance.
Furthermore, they have a pre-reproductive and a reproductive phases and quite frequently, like for holometabolism, the immature stages are well differentiated from the mature stages.
In such cases the larvae do not compete with adults, since they are found in different ecological niches, and are consequently subjected to different selective pressures (e.g.\ larvicides and adulticides factors). In response to the aforementioned, we propose:
	\begin{itemize}
		\item a continuous-time compartmental model based on a life history with two leading phases, a pre-reproductive and reproductive;
		\item a reproductive phase that includes a heredity function for Mendelian inheritance given by an autosomal gene;
		\item and selective pressures of larvicides and/or adulticides for each genotype that may affect fertility and density-dependent viability in each life phase. 
\end{itemize}

\comment{Using a heuristic principles by analogy it is possible to propose the following reductionist solution to the formulated problem:}

The life cycle of an insect may be seen analogously to a sequence of chemical modifications that a sustenance goes through.
Like in a chemical reaction network \cite{KLONOWSKI198373}, some stages of life are slower than others.
This last feature opens up the possibility of using slow manifold theory to deduce simpler inheritance models.
The modelling framework presented below is based on this principle.
More precisely, according to whether the reproductive phase is fast or slow with respect to the non-reproductive one, we obtain two distinct classes of models, which present density-dependent recruitment and mortality rates in a non-conventional way.
Simplifying the two-life phase model through slow manifold theory yields two different inheritance models, which present density-dependent recruitment and mortality rates in a non-conventional way.

These inheritance models can be seen as Mendelian general models for systems with density-dependent recruitment and mortality.
The analysis part of this paper consists in demonstrating that these two classes possess classical properties: 
in the absence of selection, they evolve in compliance with Hardy-Weinberg law; while in presence of selection and in the dominant or codominant cases, they globally converge towards the disappearance of all genotypes except the fittest homozygous one.

The paper is organised as follows.
The modelling framework is presented in Section \ref{se2}, departing from a general two life phase model.
A single locus trait heredity function that formalises Mendel's first law
is described in Section \ref{se31}.
In Sections \ref{se32} and \ref{se33}, slow manifold theory is used to deduce two classes of models describing the evolution of three genotypes of a population having inheritable attributes and density-dependent recruitment and mortality rates.
These two classes correspond respectively to the limit of fast and slow reproductive phase.
We extend in Section \ref{se4} the latter to two general classes of models, namely \eqref{SM0} ({\em Fast}) and \eqref{LM0} ({\em Slow}), which are studied afterwards.
The assumptions necessary to this study are given in Section \ref{se41}, and useful technical results are put in Section \ref{se42} (their proofs are in Appendix \ref{se10}).
Well-posedness of these models and other qualitative results are considered in Section \ref{se5}.

The asymptotic behaviour is then studied. 
The case where no fitness difference exists between the different genotypes is considered in Section \ref{se52}.
In this case, Hardy-Weinberg law is shown to hold for the considered classes of models (Theorem \ref{theo0}).
In Section \ref{se53} is studied the case of dominant and codominant selection regimes.
Asymptotic convergence to the homozygous with higher fitness is demonstrated in such conditions for models \eqref{SM0} and \eqref{LM0} (Theorem \ref{theoI}).
Last, concluding remarks are made in Section \ref{se7}.

\section{Modelling}
\label{se2}
 The aim of the present Section is to present the approach leading to the models.
 The latter are introduced in Section \ref{se4} and studied afterwards.
We begin by listing some of the main constitutive hypotheses made to obtain these models of life history for a diploid population obeying the Mendel's laws of inheritance and submitted to selection.
 \begin{itemize}
 \item
 No distinction is made between male and female individuals.
 In particular the mortality is considered identical for males and females.
 It is possible to consider this abstraction when the sex ratio is constant and fitness is independent of sex.  
 
 \item
 The mortality rates in each life phase are increasing functions of the population density.

\item
The differences in the inheritable attributes that modify the behaviour with respect to reproduction and mortality express themselves in the genotype in a single locus.
  
 \item
 There exist two different types of alleles in this locus, namely $a$ and $A$.
 \end{itemize}
 
As a consequence of the choice of two alleles in a single locus of a diploid population, we have three different genotypes, namely:
\begin{equation}
\{A,a\}\times\{A,a\}=\{(A,A),(A,a),(a,a)\}
\end{equation}
as no difference exists between the genotypes $(A,a)$ and $(a,A)$.
In the sequel, the index $i=1,2,3$ will be used to identify the {\em genotypes}, while generally speaking the index $j=A,a$ will refer to the {\em alleles}.

As said before, we are interested in the representation of a population with two life phases.
This distinction is quite simplistic, but covers broad distinctions like aquatic (subject to food and space limitation) and aerial phases, and pre-adult and adult phases.
As a starting point to represent such life history, we begin with the following compartmental models:
\begin{subequations}
	\label{c0}
	\begin{align}
	\label{c0a}
	\dot{L}_{i}  &= \omega_i \alpha_{i}(A(t))-\mu_{i}(v\t L(t))L_{i}(t)-\nu_{i}L_{i}(t)\\
	\label{c0b}
	\dot{A}_{i}  &=  \nu_{i}L_{i}(t)-\hat{\mu}_{i}(w\t A(t))A_{i}(t),  
	\end{align}
\end{subequations}
for $i=1,2,3$.
The quantities $L_{i}$  ({\em Larvae}) represent the number of individuals of the genotypes %$i$, 
$(i=1,2,3)$, in pre-adult phase, and $A_{ i}$  ({\em Adults}) the corresponding number of adult individuals.
The parameters involved in \eqref{c0} have the following meaning.

\begin{itemize}
\item
The mortality functions $\mu_{i}$ and $\hat{\mu}_{i}$ depend upon the density, through the use by the population of certain recourse (typically food or space).
We assume that these mortality rates are all {\em increasing functions of the density}.
On the other hand, we wish to allow each genotype to have its proper consumption needs.
This is done by introducing as argument of the functions $\mu_{i}$ and $\hat{\mu}_{i}$ {\em weighted sums} $v\t L$ and $w\t A$, for given {\em positive vectors} $v,w$.

\item The positive constants $\omega_i$ represent fertility rates, while the functions $\alpha_i$ account for the mechanism of Mendelian inheritance, presented in detail in Section \ref{se31}.
As seen therein, the latter are normalised by the following relation:
\begin{equation}
\label{ee3}
\forall A\in\Rset_+^3,\qquad \alpha_1(A)+\alpha_2(A)+\alpha_3(A)=A_1+A_2+A_3\ .
\end{equation}
In other words, when the fertility rates $\omega_i$ are all equal to 1, then the total number of offsprings (of all genotypes) hatched by time unit equals the total number of adults (of all genotypes).

\item
Last, the $\nu_{i}$ describe constant maturation rates from pre-adult to adult phase.

\end{itemize}

We may normalise system \eqref{c0} in order to absorb the constants $\omega_i$.
Defining the new variables 
	$\displaystyle\hL_i=\frac{L_i}{\omega_i}$, $i=1,2,3$,
one obtains a rescaled version of \eqref{c0}:
\begin{subequations}
	\label{d0}
	\begin{align}
	\label{d0a}
	\dot{\hL}_{i}  &= \alpha_{i}(A(t))-\mu_{i}(\hv\t \hL(t))\hL_{i}(t)-\nu_{i}\hL_{i}(t)\\
	\label{d0b}
	\dot{A}_{i}  &=  \hnu_{i}\hL_{i}(t)-\hat{\mu}_{i}(w\t A(t))A_{i}(t)\ ,
	\end{align}
\end{subequations}
where $\hnu_i:=\omega_i\nu_{i}$ and $\hv := \diag\{\omega_i\} v$.

The life history of organisms is quite different from one to the other.
Some mature early and reproduce quickly, while others mature late and reproduce slowly. An extreme example is the arachnids {\em Adactylidium sp.}, which are born mature and, having hatched inside their mother, mate with their brothers~\cite{elbadry1966life,gould2010panda}.
The insects {\em Ephemeroptera} constitute another extreme case: the life of an adult Mayfly is very short and has essentially the primary function of reproduction~\cite{welch1998shortest}.
We consider in the sequel the cases where one of the life phases is sensibly faster than the other one, in other words that a fast dynamics and a slow dynamics are present in \eqref{d0}.
Depending on which of the phases is faster, this assumption yields through singular perturbation \cite{Robert-Jr:1991aa} two different classes of models.
We show in Sections \ref{se32} (fast reproductive phase) and \ref{se33} (slow reproductive phase) how these two classes are obtained.

\subsection{Single locus trait inheritance} 
\label{se31}

We present here the heredity functions $\alpha_i$ used to model Mendelian inheritance.
As mentioned before, no distinction is made between male and female individuals, and we  
consider a general diploid population composed of three genotypes (in agreement with the choice of two alleles in a single locus).
The vector $x$ typically represents the adult population in equation \eqref{d0}.
We denote in the sequel $x_1(t)$ (resp.\ $x_2(t)$, resp.\ $x_3(t)$) the number of individuals of this population with genotype $(A,A)$ (resp.\ $(A,a)$, resp.\ $(a,a)$) at time $t$ in the population.
Notice that the two homozygous genotypes are represented by $x_1$ and $x_3$, while $x_2$ represents the heterozygous  genotype.

Assuming random mating, the probabilities of occurrence of all possible crosses for the two alleles in the diploid population are obtained from the following {\em Punnett Square} \cite{edward2012219}:
\begin{center}
	\begin{tabular}{c|ccc}
		$ M\backslash F $ & $ x_1 $  & $ x_2 $ & $ x_3 $  \\ 
		\hline 
		$ x_1 $ & $ 100 \% x_1 $ & $ 50 \%\ x_1, \ 50 \%\ x_2  $ & $ 100\%\ x_2 $ \\ 
		$ x_2 $ & $ 50 \%\ x_1, \ 50 \%\ x_2  $ & $ 25 \%\ x_1, \ 50 \%\ x_2, \ 25\%\ x_3 $ &  $ 50 \%\ x_2, \ 50 \%\ x_3  $ \\ 
		$ x_3 $ & $ 100 \%\ x_2 $ & $ 50 \%\ x_2, \ 50 \%\ x_3  $ & $ 100\%\ x_3 $ \\
		\hline 
	\end{tabular}   
\end{center}
The latter materialises the inheritance mechanisms expressed by the first and second Mendel's Law, namely the Law of Segregation of genes and the Law of Independent Assortment.
Note that the Punnett Square may be adapted to model other inheritance mechanisms.

Defining the vectors:
\begin{equation} 
\label{vec}
u_A:=\begin{pmatrix}
1 \\ \frac{1}{2} \\ 0
\end{pmatrix},\quad
u_a:=\begin{pmatrix}
0 \\ \frac{1}{2} \\ 1
\end{pmatrix},\quad
\bfo:=\begin{pmatrix}
1 \\ 1\\ 1
\end{pmatrix}=u_A+u_a
\end{equation}
the {\em relative frequency} of each genotype in the offspring are given respectively as
\[
\frac{(u_A\t x)^2}{(\bfo\t x)^2},\qquad
2\frac{(u_A\t x)(u_a\t x)}{(\bfo\t x)^2},\qquad \frac{(u_a\t x)^2}{(\bfo\t x)^2}\ .
\]
If the total number of offsprings appearing per time unit is equal to the total adult population (as in formula \eqref{ee3}), their  genotypic repartition is therefore given by:
\begin{equation}
\label{alpha}
\alpha_1(x):=\frac{(u_A\t x)^{2}}{\bfo\t x},\qquad
\alpha_2(x):=2\frac{(u_A\t x)(u_a\t x)}{\bfo\t x},\qquad
\alpha_3(x):=\frac{(u_a\t x)^{2}}{\bfo\t x}\ .
\end{equation}
Notice that the functions $\alpha_i$ are homogeneous of degree 1.
Defining the {\em inheritance matrices} $G_i$, $i=1,2,3$, by
\begin{equation} 
	\label{G}
	G_1 =u_Au_A\t
	=\begin{pmatrix}
	1 & 1/2 & 0\\ 1/2 & 1/4 & 0\\ 0 & 0 & 0
	\end{pmatrix},\
	G_2 =u_Au_a\t+u_au_A\t
	=\begin{pmatrix}
	0 & 1/2 & 1\\ 1/2 & 1/2 & 1/2 \\ 1 & 1/2 & 0
	\end{pmatrix},\
	G_3 =u_au_a\t
	=\begin{pmatrix}
	0 & 0 & 0\\ 0 & 1/4 & 1/2 \\ 0 &1/2 & 1
	\end{pmatrix}
	\end{equation}
one may write equivalently
\begin{equation}
\label{mof}
\alpha_i(x):=\frac{1}{\bfo\t x}x\t G_ix, \qquad i=1,2,3\ .
\end{equation}

We will also use the operator $\Aset\ :\ \Rset_+^{3} \setminus \{0_3\} \to \Rset_+^{3} \setminus \{0_3\}$ (where by definition  $0_3$ is the zero vector in $\Rset^3$) to model the mechanism of Mendelian reproduction, letting for any $x \in \Rset_+^{3} \setminus \{0_3\}$,
\begin{equation}
\label{/A(x)}
\Aset(x):= \begin{pmatrix}
\alpha_{1}(x)\\ \alpha_{2}(x) \\  \alpha_{3}(x)
\end{pmatrix}\ .
\end{equation}
Notice that a quite similar setting was introduced by Langemann {\em et al.}~\cite{langemann2013multi}, to model the evolution of herbicide resistance, in the case of single and multiple gene loci.

\subsection{Fast reproductive phase population dynamics}
\label{se32}

We consider here the case of an organism 
with a mature phase shorter than the immature phase.
As said before, the mortality rates are increasing, and this is specially true for $\hat\mu_i$ in \eqref{d0b}.
Therefore, for any {\em fixed} $\hL$, equation \eqref{d0b} possesses a unique, globally asymptotically stable, equilibrium $A(\hL)$, given implicitly by
\begin{equation}
\label{LMb}
0  =  \hnu_{i}\hL_{i}-\hat{\mu}_{i}(w\t A)A_{i},\qquad i=1,2,3\ .
\end{equation}
Equation \eqref{LMb} yields the following algebraic relationship between $L_{i}$ and $A_{i}$:
\begin{equation}
A_{i}=\frac{\hnu_{i}}{\hat{\mu}_{i}(w\t A)}\hL_{i},\qquad i=1,2,3
\end{equation}
which necessarily implies that  $ b:=w\t A $ fulfils the identity
\begin{equation}
\label{a}
b= \sum_{i=1}^{3} \frac{w_i\hnu_{i}}{\hat{\mu}_{i}(b)}\hL_{ i}
\end{equation}

Now, the mortality $\hat{\mu}_i$ is increasing with the total population, so for all $\hL \in \Rset_+^3\setminus\{0\} $, the map
\begin{equation}
b \mapsto  \sum_{i=1}^{3} \frac{w_i\hnu_{i}}{\hat{\mu}_{i}(b)}\hL_{i}
\end{equation}
is decreasing and may be inverted, providing a unique solution, denoted $b^*(\hL)$, to equation \eqref{a}.
For any given nonnegative vector $\hL$, the unique solution of \eqref{LMb} is then given as
\begin{equation}
A_i = \frac{\hnu_{i}}{\hat{\mu}_{i}(b^*(\hL))}\hL_{i},\qquad i=1,2,3\ .
\end{equation}

In the limiting case where the duration of the mature phase is sensibly faster than the immature one, one may approximate equation \eqref{d0} through slow manifold theory \cite{Robert-Jr:1991aa}.
The asymptotic evolution is then expressed by the algebro-differential system
\begin{equation}
	\label{LM}
	\dot{\hL}_{i}  = \alpha_{i}(A)-\mu_{i}(\hv\t \hL)\hL_{i}-\nu_{i}\hL_{i},\qquad
A_i = \frac{\hnu_{i}}{\hat{\mu}_{i}(b^*(\hL))}\hL_{i},\qquad i=1,2,3\ .
\end{equation}
Defining for any nonnegative scalar $b$, $\displaystyle m_{i}(b):=\frac{\hnu_{i}}{\hat{\mu}_{i}(b)}$, one then ends up with the system
\begin{subequations}
	\label{LWM0}
\begin{equation}
	\label{LWM0a}
	\dot{\hL}_{i}=\alpha_{i}\left(\diag\{m_{i}(b^*(\hL))\} \hL\right) -(\nu_{i} + \mu_{i}(\hv\t \hL))\hL_{i},  \qquad i=1,2,3
\end{equation}
where, for any $L\in\Rset_+^3$, $b^*(L)$ is defined implicitly by the identity:
\begin{equation}
	\label{LWM0b}
b^*(L) = \sum_{i=1}^3 w_i m_i(b^*(L)) L_i\ .
\end{equation}
\end{subequations}
In \eqref{LWM0a}, $\diag\{m_{i}(b^*(\hL))\}$ denotes the {\em diagonal} matrix formed with the scalar coefficients $m_{i}(b^*(\hL))$, $i=1,2,3$.
Notice that by construction the functions $m_i$ present in \eqref{LWM0} are {\em decreasing}.

\subsection{Slow reproductive phase population dynamics}
\label{se33}

Consider now the opposite situation, of an adult phase comparatively much longer than the pre-adult one.
In this case, the same argument than before leads similarly to consider instead of system \eqref{d0} the algebro-differential model:
\begin{equation}
	\label{A}
	0  =  \alpha_{i}(A)-\mu_{i}(\hv\t \hL)\hL_{ i}-\nu_{i}\hL_{i},\qquad
	\dot{A}_{i} =  \hnu_{i}\hL_{i}-\hat{\mu}_{i}(w\t A)A_{i},\qquad i=1,2,3
\end{equation}

The algebraic relationship in equation \eqref{A} provides the identities
\begin{equation}
\label{nu}
\hL_{i}=\frac{1}{\nu_{i}+\mu_{i}(\hv\t \hL)}\alpha_{i}(A),\qquad i=1,2,3
\end{equation}
and thus $ b:=\hv\t \hL $ necessarily fulfils the condition:
\begin{equation}
\label{b}
b=\sum_{i=1}^{3} \frac{\hv_i}{\nu_{i}+\mu_{i}(b)}\alpha_{i}(A)
\end{equation}

From the increasingness of the mortality functions, one deduces as before that, for all $ A \in \Rset_+^3\setminus\{0\} $, the right-hand side of \eqref{b} is decreasing, and this relation may thus be inverted. 
This operation yields a unique solution to equation \eqref{b}, which is denoted $b^*(\Aset(A))$ (using $\Aset$ defined in  \eqref{/A(x)}).
With this, for any nonnegative $A$, \eqref{nu} has a unique solution $\hat L(A)$, which writes
\begin{equation}
\hL_{i}=\frac{1}{\nu_{i}+ \mu_{i}(b^*(\Aset(A)))}\alpha_{i}(A), i=1,2,3\ .
\end{equation}
Defining here for any nonnegative scalar $b$, (decreasing) functions $\displaystyle m_{i}(b):=\frac{\hnu_i}{\nu_{ i}+\mu_{i}(b)}$, one obtains finally the model:
\begin{subequations}
	\label{A0}
	\begin{equation}
	\label{A0a}
	\dot{A_{i}}=\alpha_{i}(A)m_{ i}(b^*(\Aset(A)))-\hat{\mu}_{ i}(w\t A)A_{i},  \quad i=1,2,3
	\end{equation}
where, for any $A\in\Rset_+^3$, $b^*(A)$ is defined implicitly by
\begin{equation}
	\label{A0b}
b^*(A) = \sum_{i=1}^3 \frac{\hat v_i}{\hat\nu_i} m_i(b^*(A)) A_i\ .
\end{equation}
\end{subequations}
	
Equation \eqref{A0} is quite similar to, but different from, equation \eqref{LWM0} obtained in the case of a short reproductive phase.

As can be seen from the derivations of \eqref{LWM0} and \eqref{A0}, monotonicity assumptions on the mortality rates are necessary, in order to obtain the perturbed systems.
For this reason, we postpone the formal statement of the models studied in the sequel to Section \ref{se4}, where all the assumptions are introduced.

\section{Preliminaries}
\label{se4}

\subsection{Assumptions}
\label{se41}

We now introduce the two general classes of equations studied in this paper.
The latter extend in particular the systems \eqref{LWM0} and \eqref{A0} obtained previously as models for fast and slow reproductive phase populations.
The notations are inspired from these preliminary examples and generalised.

We introduce first functions $m_i$ and $\mu_i$ to model {\em recruitment} and {\em removal} rates, and $v, w$ two vectors permitting to define their arguments.
The following series of assumptions will be made on these objects.

\begin{assumption}
	\label{ass0}
	For any $i=1,2,3$, the functions $m_i:\Rset_+\rightarrow\Rset_+$ are locally Lipschitz decreasing functions; the functions  $\mu_i:\Rset_+\rightarrow\Rset_+$ are locally Lipschitz increasing functions.
	The vectors $v,w\in\Rset^3$ have positive components.
\end{assumption}
Assumption \ref{ass0} assumes that the recruitment rates are decreasing functions, while the removal rates are increasing functions.
As a central consequence, this allows to give sense to the function $b^*$ that naturally appeared in Sections \ref{se32} and \ref{se33}, as shown by the following result, whose proof is in Appendix \ref{app1}.
The values $v_i$ that appear in the statement are the components of the vector $v$ introduced in Assumption \ref{ass0}.

\begin{lemm} 
\label{a0}
	Let Assumption {\rm\ref{ass0}} holds.
	Then, for any $x \in \mathbb{R}_{+}^{3}$, there exists a unique solution  $b^*(x)\in\Rset_+$, to the scalar equation 
	\begin{equation} 
	\label{a*}
	b=\sum_{i=1}^{3} v_im_i(b)x_i.
	\end{equation}
	Moreover, $b^*(x)>0 $ if  $ x \in \mathbb{R}_{+}^{3}\setminus \{0_3\} $, $b^*(0_3)=0 $ and $b^*: \mathbb{R}_{+}^{3}\rightarrow\mathbb{R}_+  $ has the same regularity than the functions $m_i$. 
\end{lemm}

\comment{Also, one can verify that for any $x \in \Rset_+^{3}$ exist $ \underline{a}, \bar{a} : \Rset_+ \rightarrow \Rset_+ $ such that: 
\begin{equation}
\forall x \in \Rset_{+}^3 \qquad \underline{a}(\bfo\t x) \leq b^*(x) \leq \bar{a}(\bfo\t x)
\end{equation} Indeed, from \eqref{a*} one has 
\begin{equation}
\forall x \in \Rset_{+}^3 \qquad \min_{i} \{ m_{i}(b^*(x))\}\bfo\t x \leq b^*(x) \leq \max_{i} \{ m_{i}(b^*(x))\}\bfo\t x
\end{equation} one achieved with implicitly define $ \underline{a}$ and $\bar{a} $ for:
\begin{align}
\label{aa}
\underline{a}(\bfo\t x) & :=  \min_{i} \{ m_{i}(b^*(x))\} \bfo\t x\\
\bar{a}(\bfo\t x)      & :=  \max_{i} \{ m_{i}(b^*(x))\} \bfo\t x
\end{align}}
Notice that the map $b^*$ depends upon the vector $v$.\\

With this done, we are finally in position to introduce the two classes of models studied in this paper.
For any $x \in \mathbb{R}_+^{3}$ and $z \in \mathbb{R}_{+}$, define the {\em positive diagonal matrices} $M(x)$ and  $\mu(z)$ by: 
\begin{equation}
	\label{M,mu}
	M(x)  = \diag\{m_i(b^*(x))\},\qquad
	\mu(z) = \diag\{\mu_{i}(z)\}.
\end{equation}
The two classes of models we will be interested in are expressed as follows:
{\renewcommand{\theequation}{\bf F} 
	\begin{equation}
	\label{SM0}
	\dot{x} = \Aset(M(x)x)-\mu(w\t x)x\ ,
	\end{equation}}
\vspace{-.4cm}
{\renewcommand{\theequation}{\bf S}
	\begin{equation}
	\label{LM0}
	\dot{x} = M(\Aset(x))\Aset(x)-\mu(w\t x)x\ .
	\end{equation}}
Recall that the map $\Aset$ is defined in \eqref{alpha}-\eqref{/A(x)}.
In developed form, these equations write respectively
{\renewcommand{\theequation}{\bf F.}
	\begin{subequations}
		\begin{align}
		\label{Sa}
		\dot x_1 & =  \frac{(u_A\t M(x)x)^2}{\bfo\t M(x)x} - \mu_1(w\t x)x_1\\
		\label{Sb}
		\dot x_2 & =  2\frac{(u_A\t M(x)x)(u_a\t M(x)x)}{\bfo\t M(x)x} - \mu_2(w\t x)x_2\\
		\label{Sc}
		\dot x_3 & =  \frac{(u_a\t M(x)x)^2}{\bfo\t M(x)x} - \mu_3(w\t x)x_3
		\end{align}
\end{subequations}}
and
{\renewcommand{\theequation}{\bf S.}
	\begin{subequations}
		\begin{align}
		\label{La}
		\dot x_1 & =\frac{(u_A\t x)^2}{\bfo\t x} m_1(b^*(\Aset(x)))  - \mu_1(w\t x)x_1\\
		\label{Lb}
		\dot x_2 & = 2\frac{(u_A\t x)(u_a\t x)}{\bfo\t x}m_2(b^*(\Aset(x)))  - \mu_2(w\t x)x_2\\
		\label{Lc}
		\dot x_3 & = \frac{(u_a\t x)^2}{\bfo\t x}m_3(b^*(\Aset(x)))  - \mu_3(w\t x)x_3\ .
		\end{align}
\end{subequations}}
Manifestly, equations \eqref{SM0} and \eqref{LM0} contain as particular cases the equations \eqref{LWM0} and \eqref{A0}, obtained by applying singular perturbation to the % normalised 
normalized system \eqref{d0}.
One shows easily that the system \eqref{LWM0} that emerged in the case of fast reproductive phase (Section \ref{se32}) is system \eqref{SM0} with the recruitment rates $m_i$ and the mortality rates $\mu_i$
\addtocounter{equation}{-4}
\begin{equation}
\frac{\omega_i\nu_i}{\hat\mu_i(\cdot)}\qquad\text{ and }\qquad \nu_i+\mu_i(\cdot)\ .
\end{equation}
 The variable $x$ represents the pre-adult populations in the original model.
On the other hand,
system \eqref{A0} appeared in the slow reproductive phase (Section \ref{se33}) corresponds to \eqref{LM0} with recruitment and mortality rates
\begin{equation}
\frac{\omega_i\nu_i}{\nu_i+\mu_i(\cdot)}\qquad\text{ and }\qquad \hat\mu_i(\cdot)\ ,
\end{equation}
 and $x$ now represents the adult populations.

The following supplementary assumptions will be used to study adaptation.

\begin{assumption}
	\label{ass1}
	The functions $m_i$, $\mu_i$ verify
	\begin{equation}
	\label{ee24}
	\forall x\in\Rset_+^3,\quad
	m_1(b^*(x)) \geq m_2(b^*(x)) \geq m_3(b^*(x))
	\qquad \text{ and } \qquad
	\forall z\geq 0,\quad
	\mu_1(z)\leq \mu_2(z) \leq \mu_3(z)
	\end{equation}
Moreover, for any $x\in\Rset_+^3$,
\begin{itemize}
\item
$m_1(b^*(x))-m_3(b^*(x))+\mu_3(w\t x)- \mu_1(w\t x)>0$ for system \eqref{SM0};
\item
$m_1(b^*(\Aset(x)))-m_3(b^*(\Aset(x)))+\mu_3(w\t x)- \mu_1(w\t x)>0$ for system \eqref{LM0}.
\end{itemize}
\end{assumption}

\begin{assumption}
	\label{ass2}
	
	The following inequality holds
	\begin{equation}
	m_1(0) > \mu_1(0)
	\end{equation}
\end{assumption}

\begin{assumption}
	\label{ass3}
	For any $x\in\Rset_+^3\setminus\{0_3\}$, the following limits exist and verify 
	\begin{equation}
	0\leq \lim_{\lambda\to +\infty} m_i(\lambda x) < \lim_{z\to +\infty} \mu_i(\lambda) \leq +\infty, \qquad i=1,2,3.
	\end{equation}
\end{assumption}
 
Assumption \ref{ass1} imposes a relative ordering of the fitnesses of the three genotypes, associated with corresponding increasing mortality rates and decreasing recruitment rates.
Two distinct important situations are covered: the cases of {\em codominance} or {\em incomplete dominance}, when the fitness of the two  homozygotes and of the heterozygote are strictly ordered according to \eqref{ee24} (e.g.\ $m_1>m_2>m_3$ and $\mu_1<\mu_2<\mu_3$); and the case of {\em dominance}, where the heterozygote has the same fitness than one of the two homozygotes (e.g.\ $m_1=m_2=m_3$, and $\mu_1<\mu_2=\mu_3$ or $\mu_1=\mu_2<\mu_3$; or $\mu_1=\mu_2=\mu_3$, and $m_1>m_2=m_3$ or $m_1=m_2>m_3$).
Notice that the {\em selectively neutral} case, where the fitness of all genotypes is equal, is excluded by the positivity requirement contained in Assumption \ref{ass1}.

Assumption \ref{ass2} ensures that (at least) genotype 1, with higher fitness, is viable.
Last, Assumption \ref{ass3} indicates that in any sufficiently large population, the mortality rates are %larger 
greater than the recruitment rates for each genotype, and this overpopulation effect will limit the population growth.

\subsection{Technical lemmas}
\label{se42}

Before proceeding to the analysis results in Section \ref{se5}, we state now two technical lemmas useful in the sequel.
By definition, $e_i$, $i=1,2,3$, represent the vectors of the canonical basis in $\Rset^3$.

\begin{lemm} 
\label{ppA}
	For any $ x \in \mathbb{R}_{+}^{3}\setminus \{0_3\}$ and any $ \lambda > 0$, the following properties are fulfilled 
	\begin{enumerate}[i.]
		\item $ u_j\t \Aset(x)= u_j\t x, \ j=A,a $; \label{ppi}
		\item $ \bfo\t \Aset(x)= \bfo\t x $;\label{ppii}
		\item $ \Aset(\lambda x)= \lambda \Aset(x) $;\label{ppiii}
		\item $ \Aset(e_i)=e_i$, $i=1,3$.
		\label{ppiv}		
	\end{enumerate}
\end{lemm} 

Recall that the map $\Aset$ is defined in \eqref{alpha}-\eqref{/A(x)}, and that the vectors $u_A, u_a$ and $\bfo$ come from \eqref{vec}.

\begin{lemm}
	\label{le0}
	Let $x \in \mathbb{R}_{+}^{3}\setminus \{0\} $.
	
	\begin{enumerate}[i.]
		\item
		The following properties are verified
		\begin{equation}
		\label{a*toinf}
		\lim_{\lambda\to +\infty} b^*(\lambda x)=+\infty, 
		\end{equation} 
		
		\begin{equation}
		\label{milessmui}
		\lim_{\lambda\to +\infty} m_i(b^*(\lambda x))< \lim_{\lambda\to +\infty} \mu_i(\lambda), \quad  i=1,2,3.
		\end{equation}
		Moreover, for both limits the convergence is uniform in the set $ \{x \in \mathbb{R}_{+}^{3}: w\t x=1 \} $.
		\label{pppi}
		\item
		The map $\lambda \mapsto b^*(\lambda x)$ is increasing on $[0,+\infty)$.
		\label{pppii}
		
	\end{enumerate}
	
\end{lemm}

Proofs of Lemmas \ref{ppA} and \ref{le0} are given in Appendices \ref{app2} and \ref{app3}.
 Notice that the map $x\mapsto b^*(x)$ is not itself increasing.

\section{Well-posedness and qualitative results}
\label{se5}

We provide here the first results concerning the solutions of systems \eqref{SM0} and \eqref{LM0}.
At first, well-posedness of these systems of nonlinear ordinary differential equations is proved in Section \ref{se51}.
The qualitative properties that relate to the existence of {\em monomorphic} and {\em polymorphic states}, and {\em trajectories} are studied in section 4.2.
Last, we show in Section \ref{se518} that quite natural notions of {\em mean allelic recruitment rate} and {\em mean allelic mortality rate} may be introduced for each system, which will prove quite useful in establishing the forthcoming asymptotic results.

\subsection{Well-posedness of the models}
\label{se51}

Well-posedness for the models \eqref{SM0} and \eqref{LM0} is supplied by the following statement.

\begin{theo}[Well-posedness and boundedness of the solutions]
	\label{le01}
	Assume Assumption {\rm\ref{ass0}} is fulfilled.
	Then, for any nonnegative initial condition, there exists a unique solution $x$ of \eqref{SM0}, resp.\ \eqref{LM0}.
 	Its coordinates are nonnegative for any $t\geq 0$.
	
	If moreover Assumption {\rm\ref{ass3}} is fulfilled, then
	\begin{equation}
	\label{eq}
	0 \leq \liminf_{t\to +\infty} \bfo\t x(t) \leq \limsup_{t\to +\infty} \bfo\t x(t) \leq c_{\ref{SM0}}^*,\quad
	\text{ resp.\ } c_{\ref{LM0}}^*
	\end{equation}
where by definition $c_{\ref{SM0}}^*$, resp.\ $c_{\ref{LM0}}^*$, is the largest $c\geq 0$ such that
	\begin{equation}
	\label{eqq}
	\max\left\{\max_i\dfrac{m_i(b^*(cx))}{\mu_i(c)}\ : \ w\t x=1
	\right\} \leq 1,\ \text{ resp.\ }
	\max\left\{\max_i\dfrac{m_i(b^*(c\Aset(x)))}{\mu_i(c)}\ : \ w\t x=1
	\right\} \leq 1
	\end{equation}
\end{theo}

\begin{proof}
\mbox{}
	
\noindent $\bullet$
The equations \eqref{SM0} and \eqref{LM0} are meaningful as soon as Assumption \ref{ass0} holds, as the latter permits to define $b^*$ (through Lemma \ref{a0}).
	The well-posedness of both systems comes from the Lipschitzness of their right-hand side.
	The fact that the trajectories do not escape the nonnegative quadrant comes from the fact that $\dot x_i(t)\geq 0$   
	whenever $x_i(t)=0$, $i=1,2,3$.
	
\noindent $\bullet$
	We first demonstrate that the definition of $c_{\ref{SM0}}^*$ as stated in the statement is meaningful.
	Indeed, for any $ x \in \Rset_+^{3} \setminus \{0_3\}$,
	due to Lemma \ref{le0} the map 
	\begin{equation}
	c \mapsto \max_{i=1,2,3}\frac{m_i(b^*(cx))}{\mu_i(c)}
	\end{equation}
 is decreasing and, due to Assumption \ref{ass3}, verifies:
	\begin{equation}
	\lim_{c\to +\infty} \max_{i=1,2,3}\frac{m_i(b^*(cx))}{\mu_i(c)}<1.
	\end{equation}
	As the convergence  of this limit is uniform in the set $ \{x \in \mathbb{R}_+^{3}: w\t x=1  \} $, the map 
	\begin{equation}
	c \mapsto \max\left\lbrace  \max_{i=1,2,3}\dfrac{m_i(b^*(cx))}{\mu_i(c)}: w\t x =1 \right\rbrace 
	\end{equation} admits values smaller than $ 1 $ for sufficiently large $ c>0  $.
	Therefore, $c_{\ref{SM0}}^*$ as given in the statement is well defined. 
	
	Slight adaptation of the same argument demonstrates that $c_{\ref{LM0}}^*$ too is well defined.
	
\noindent $\bullet$
	Summing up the three equations in \eqref{SM0}, we have 
	\begin{eqnarray}
	\label{sumdotx} 
	\bfo\t \dot{x}
	& = &
	\nonumber
	\bfo\t \Aset(M(x)x)-\bfo\t \mu(w\t x)x\\
	& = &
	\nonumber
	\bfo\t M(x)x-\bfo\t \mu(w\t x)x \qquad (\text{by Property {\it\ref{ppii}} in Lemma \ref{ppA}})\\
	& =  &
	\nonumber
	\sum_{i=1}^{3} (m_i(b^*(x))-\mu_i(w\t x))x_i\\
	& \leq  &
	\left(
	\max_i\{m_i(b^*(x))-\mu_i(w\t x)\}
	\right) \bfo\t x, \quad i=1,2,3.
	\end{eqnarray}
	 The vector $\frac{x}{w\t x}=1$ is normalised, in the sense that $w\t \frac{x}{w\t x}=1$.
	As $m_i(b^*(x))=m_i(b^*(w\t x \frac{x}{w\t x}))$, one gets by definition of $c_{\ref{SM0}}^*$ that  $\max_i\{m_i(b^*(x))-\mu_i(w\t x)\} < 0 $ whenever $w\t x > c_{\ref{SM0}}^*$.

Essentially the same argument provides the corresponding result for system \eqref{LM0}, and this proves \eqref{eq} and \eqref{eqq}, and achieves the demonstration of Theorem \ref{le01}.
\end{proof}

\subsection{Monomorphism and polymorphism}
\label{se515}

Variability is essential for the selection to operate in a population.
With this in view, we introduce some related notions.

\begin{defi}[Monomorphic, polymorphic and ``holomorphic'' states]
	\label{xej}
	A {\em nonzero} population state $x\in\Rset_+^3$ is called {\em monomorphic} if it consists of a single homozygous genotype, that is $x=ce_i$ for $i=1$ or $3$ and a certain $c>0$.
	Otherwise it is called {\em polymorphic}, and ``holomorphic'' if all the genotypes are present.
\end{defi}

We exhibit now the possible values of a homozygous equilibrium point.

\begin{lemm}[Monomorphic equilibria] 
\label{monoEqui}
		Assume Assumptions {\rm\ref{ass0}} and  {\rm\ref{ass3}} are fulfilled.
Let $i\in\{1,3\}$.
If $ m_i(0)>\mu_i(0) $, then there exists a unique positive solution $c_i^*\in \Rset_+ $ to the scalar equation 
		\begin{equation}
		\label{c_i*}
		m_i(b^*(c_ie_i))=\mu_i(c_iw_i)\ .
		\end{equation}
		If $ m_i(0)\leq\mu_i(0) $, we define $ c_i^*:=0 $. 
\end{lemm}

\begin{proof}
	Notice that the functions defined on $\Rset^+$ by $c\mapsto m_i(b^*(ce_i))$, $i=1,2,3$, are decreasing,
	due to Assumption \ref{ass0} and property {\it\ref{pppii}} in Lemma \ref{le0}, while the functions $c\mapsto \mu_i(cw_i)$, $i=1,2,3$, are increasing.
	Therefore for $ m_i(0)>\mu_i(0)$, \eqref{c_i*} admits a unique, positive, solution.
\end{proof}

The following lemma shows that the trajectories of \eqref{SM0} and \eqref{LM0} initially originated from a monomorphic population converge to the corresponding monomorphic equilibrium.
\begin{lemm}[Monomorphic trajectories]
	\label{le1}
Assume Assumptions {\rm\ref{ass0}} and  {\rm\ref{ass3}} are fulfilled.
Then any trajectory originating from a monomorphic state, say of homozygous $i\in\{1,3\}$, stays monomorphic and converges towards the corresponding equilibrium point $ c^*_ie_i $, with $c^*_i$ given by Lemma {\rm \ref{monoEqui}}. 
\end{lemm}
\begin{proof}
	Clearly when only one allele is initially present (i.e.\ when the heterozygous genotype and one of the two homozygous ones are initially absent), this property remains true throughout time.
	The dynamics of \eqref{SM0} or \eqref{LM0} then occur in the one-dimensional space that corresponds to this homozygous genotype, and it is easily shown that the evolution obeys the law:
	\begin{equation}
	\label{dotxj}
	\dot x_i = \left(
	\vphantom{\hat A}
	m_i(b^*(x_ie_i))-\mu_i(w_ix_i)
	\right)x_i\ .
	\end{equation} 
	In the case where $c_i^*$ is such that $ m_i(b^*(c_i^*e_i))-\mu_i(w_ic_i^*)=0 $ (that is when $ i=1 $, or when $ i=3 $ and $ m_3(0)>\mu_3(0) $), then $ c^*_i$ is the unique globally asymptotically stable equilibrium of this system. 
	
	When $m_3(0)\leq\mu_3(0) $, all trajectories converge towards the globally asymptotically equilibrium $ 0=c^*_3e_3 $ (as in this case $ c^*_3=0 $).   
\end{proof}

We now consider trajectories of \eqref{SM0} and \eqref{LM0} such that $u_A\t x(0) \neq 0$ {\em and} $u_a\t x(0) \neq 0$.
The following result shows that in such cases all alleles, {\em but also all genotypes}, are present for positive times: extinction of a genotype (in the selection case) may only occur asymptotically in time.
In other words, {\em polymorphic} trajectories are also ``holomorphic'' trajectories: immediately after the initial time, they contain all genotypes.

\begin{lemm}[Polymorphic trajectories] 
	\label{le15}
	Assume Assumption {\rm\ref{ass0}} is fulfilled.
	Then for any trajectory such that $u_A\t x(0)\neq 0$ and $u_a\t x(0)\neq 0$,
	\begin{equation}
	\forall t> 0,\qquad x_i(t) > 0,\qquad i=1,2,3\ .
	\end{equation}
\end{lemm}

\begin{proof}[Proof of Lemma  \rm\ref{le15}]
	If $x_i(0)= 0$ for $i=1$ or 3, and $x_2(0)> 0$, then the corresponding derivative $\dot x_i$ is positive at time $t=0$, and thus $x_i$ takes on positive values at the right of $0$. Similarly, if $x_i(0)> 0$ for $i=1$ and 3, and $x_2(0)= 0$, the same occurs for the derivative $\dot x_2$, and $x_2$ is also positive at the right of 0.
	
	On the other hand, one sees that, for any $t\geq 0$ and $i=1,2,3$,
	\begin{equation}
	\label{dotx+mu} 
	\dot{x_i}  \geq -\mu_i(w\t x)x_i\ . 
	\end{equation}
	Therefore
	\begin{equation}
	\forall t\geq t' \geq 0,\qquad
	x_i(t) \geq x_i(t') e^{-\int_{t'}^t \mu_i(w\t x(s)) ds}
	\end{equation} which is positive whenever $x_i(t')>0$.
	This establishes the desired inequality.
\end{proof}

In view of Lemma \ref{le15}, we put

\begin{defi}[Polymorphic trajectories]
	Any trajectory such that  $u_A\t x(0) \neq 0$ and $u_a\t x(0) \neq 0$ is called a {\em polymorphic trajectory}.
\end{defi}

Due to Lemma \ref{le15}, any polymorphic trajectory is constituted of holomorphic states, except possibly at its initial point.

\subsection{Mean allelic mortality and recruitment rates}
\label{se518}

In order to study selection in Section \ref{se53}, we associate here to each of the systems \eqref{SM0} and \eqref{LM0} two {\em mean allelic rates}, which are defined {\em at any polymorphic state}.
Due to Lemma \ref{le15} and property {\it\ref{ppi}} in Lemma \ref{ppA}, these rates are well defined along any polymorphic trajectory.

\begin{defi}[Mean allelic mortality and recruitment rates]
\label{de99}
\mbox{}

\noindent
{\bf $\bullet$ System \eqref{SM0}.}
For any polymorphic state $x \in \Rset_+^3$, define the {\em mean allelic recruitment rates} $\tm_{\ref{SM0},j}(x)$:
\begin{subequations}
\label{eqmA5}
	\begin{equation}
	\tm_{\ref{SM0},j}(x) := \frac{u_j\t M(x)x}{u_j\t x},\qquad j=A,a
	\label{eqmAR5}
	\end{equation}
and the {\em mean allelic mortality rates} $\tmu_{\ref{SM0},j}(x)$:
	\begin{equation}
	\tmu_{\ref{SM0},j}(x) := \frac{u_j\t \mu(w\t x)x}{u_j\t x},\qquad j=A,a.
	\label{eqmuAR5}
	\end{equation}
 \end{subequations}
  
\noindent
{\bf $\bullet$ System \eqref{LM0}.}
For any polymorphic state $x \in \Rset_+^3$, define the {\em mean allelic recruitment rates} $\tm_{\ref{LM0},j}(x)$:
\begin{subequations}
\label{eqmA8}
	\begin{equation}
	\tm_{\ref{LM0},j}(x) := \frac{u_j\t M(\Aset(x))\Aset(x)}{u_j\t x},\qquad j=A,a
	\label{eqmAR8}
	\end{equation}
	and the {\em mean allelic mortality rates} $\tmu_{\ref{LM0},j}(x)$:
	\begin{equation}
	\tmu_{\ref{LM0},j}(x) := \frac{u_j\t \mu(w\t x)x}{u_j\t x},\qquad j=A,a.
	\label{eqmuAR8}
	\end{equation}
 \end{subequations}

\end{defi}   
 
 The fundamental interest of the previous definitions is to allow writing the evolution of the allelic populations {\em along any polymorphic trajectory} as:
\begin{subequations}
\label{ady}
\begin{eqnarray}
\nonumber
u_j\t \dot x
& = & u_j\t \left(
\Aset(M(x)x)-\mu(w\t x)x
\right)
= u_j\t M(x)-  \tmu_{\ref{SM0},j}(x) u_j\t x\\
& = &
\label{adya}
\left(
\tm_{\ref{SM0},j}(x) - \tmu_{\ref{SM0},j}(x)
\right) u_j\t x,\qquad j=A,a
\end{eqnarray}
for system \eqref{SM0}; and similarly for system \eqref{LM0}:
 \begin{eqnarray}
u_j\t \dot x
& = &
\nonumber
u_j\t \left(
M(\Aset(x))\Aset(x)-\mu(w\t x)x
\right)
= \left(
\frac{u_j\t M(\Aset(x))\Aset(x)}{u_j\t \Aset(x)} - \tmu_{\ref{LM0},j}(x)
\right) u_j\t x\\
\label{adyb}
& = &
\left(
\tm_{\ref{LM0},j}(x) - \tmu_{\ref{LM0},j}(x)
\right) u_j\t x,\qquad j=A,a
\end{eqnarray}
\end{subequations}
(Property {\it\ref{ppi}} of Lemma \ref{ppA} was used in the previous deductions.)

A key characteristic of the objects introduced in Definition \ref{de99} is summarised in the following lemma.

\begin{lemm}[Ordering of the mean allelic rates]
\label{oma}
Assume Assumptions {\rm\ref{ass0}} and {\rm\ref{ass1}} are fulfilled.
Then for any polymorphic state $x \in \Rset_+^3$, one has
\begin{subequations}
\label{eq102}
\begin{gather}
m_1(b^*(x)) \geq \tm_{\ref{SM0},A}(x) \geq m_2(b^*(x)) \geq \tm_{\ref{SM0},a}(x) \geq m_3(b^*(x))\ ,\\
m_1(b^*(\Aset(x))) \geq \tm_{\ref{LM0},A}(x) \geq m_2(b^*(\Aset(x))) \geq \tm_{\ref{LM0},a}(x) \geq m_3(b^*(\Aset(x)))
\end{gather}
\end{subequations}
and
\begin{subequations}
\label{eq10}
\begin{gather}
\mu_1(w\t x) \leq \tmu_{\ref{SM0},A}(x) \leq \mu_2(w\t x) \leq \tmu_{\ref{SM0},a}(x) \leq \mu_3(w\t x)\ ,
\label{eq10a}\\
\mu_1(w\t x) \leq \tmu_{\ref{LM0},A}(x) \leq \mu_2(w\t x) \leq \tmu_{\ref{LM0},a}(x) \leq \mu_3(w\t x)
\label{eq10b}
\end{gather}
\end{subequations}
\end{lemm}

\begin{proof}
For some $j\in\{A,a\}$ consider e.g.\ the map $\tmu_{\ref{SM0},j}$.
One has
\begin{equation}
\tmu_{\ref{SM0},j}(x) := \frac{u_j\t \mu(w\t x)x}{u_j\t x}
\end{equation}
and thus
\begin{equation}
\tmu_{\ref{SM0},1}(x)
= \frac{\mu_1(w\t x)x_1 + \frac{1}{2}\mu_2(w\t x)x_2}{x_1 + \frac{1}{2}x_2},\qquad
\tmu_{\ref{SM0},3}(x)
= \frac{\mu_3(w\t x)x_3 + \frac{1}{2}\mu_2(w\t x)x_2}{x_1 + \frac{1}{2}x_2}
\end{equation}
and Assumption \ref{ass1} yields immediately \eqref{eq10a}.
The three other formulas are proved in the same way.
\end{proof}

\comment{
	Due to the positivity hypothesis in Assumption \ref{ass1}, one has the following result.
	
	\begin{lemm}
		\label{le77}
		Let $x$ be any polymorphic state such that
		\begin{equation}
		\label{eq86}
		\tmu_A(x) = \tmu_A(x) \qquad \text{ \em and } \qquad \tm_A(x) = \tm_a(x)
		\end{equation}
		Then $x_1=x_3=0$.
	\end{lemm}
	
	\begin{proof}
		Let $x$ be a polymorphic state fulfilling \eqref{eq86}.
		Assume first that $m_1(b^*(x)) > m_3(b^*(x))$.
		Then, in order to have $\tilde m_1(x) = \tilde m_3(x)$, one must necessarily have $x_1=x_3=0$, see formula \eqref{eqmAR5}.
		
		If now $m_1(b^*(x)) = m_3(b^*(x))$, then $\mu_1(\bfo\t x) < \mu_1(\bfo\t x)$, due to the strict inequality stated in Assumption \ref{ass1}.
		Now, $\tmu_A(x) = \tmu_A(x)$ also implies that $x_1=x_3=0$, see \eqref{eqmuAR5}.
		
		In conclusion, when \eqref{eq86} is fulfilled, then $x_1=x_3=0$ in any case.
		This demonstrates Lemma \ref{le77}.
	\end{proof}
}

\section{Analysis of the selectively neutral case}
\label{se52}

One considers here the {\em selectively neutral} case, where recruitment and mortality rates are identical for all genotypes.
Write in this case $\mu_\text{\rm sn}:=\mu_i$, $m_\text{\rm sn}:=m_i$ for $i = 1,2,3$.
As a consequence of Assumption \ref{ass0}, $m_\text{\rm sn}$ is decreasing and $\mu_\text{\rm sn}$ increasing.
Equation \eqref{a*} here writes
\begin{equation}
b = m_{\rm sn} (b) v\t x
\end{equation}
and its unique scalar solution clearly depends upon $x$ only through the quantity $v\t x$.
To emphasize this fact, it is denoted $b_{\rm sn}(v\t x)$, rather than $b^*(x)$.
As a direct consequence of Lemma \ref{a0}, $b_{\rm sn}$ is null at zero and takes on positive values otherwise.
Also, it comes from the second point of Lemma \ref{le0} that $b_{\rm sn}$ is increasing.
The following result for \eqref{SM0} and \eqref{LM0} then presents no difficulty. 

\begin{lemm} 
	\label{le44}
	Assume Assumptions  {\rm\ref{ass0}, \ref{ass2}}
	and {\rm\ref{ass3}} are fulfilled.
	For any $x\in\Rset_+^3\setminus\{0\}$, there exists a unique solution, denoted $c_{\rm sn}^*(x)$ in the sequel, of the equation
	\begin{equation}
	\label{eq19}
	m_{\rm sn}(b_{\rm sn}(c v\t x))=\mu_{\rm sn}(c w\t x)
	\end{equation} 
\end{lemm}
\begin{proof}
	Due to Assumption \ref{ass0}, the map $c\mapsto m_{\rm sn}\circ b_{\rm sn}(c v\t x)$ is decreasing, while the map $c\mapsto \mu_{\rm sn}(c w\t x)$ is increasing.
	On the other hand, $m_{\rm sn}(b_{\rm sn}(0)) = m_{\rm sn}(0) > \mu_{\rm sn}(0)$ from Assumption \ref{ass2};
	and $0\leq \displaystyle \lim_{z\to +\infty} m_{\rm sn}(z) < \lim_{z\to +\infty} \mu_{\rm sn}(z) \leq +\infty$ due to Assumption \ref{ass3}.
	This demonstrates Lemma \ref{le44}.
\end{proof}

Both systems \eqref{SM0} and \eqref{LM0} now boil down to the equation
\begin{equation}
\label{eqsf}
\dot x = m_{\rm sn}(b_{\rm sn}(v\t x)) \Aset(x) -\mu_{\rm sn}(w\t x) x\ .
\end{equation}
The asymptotic behaviour of the solutions of \eqref{eqsf} is completely described by the following result.
In particular, its proof makes clear that the total population follows a variant of Verhulst's logistic growth equation \cite{wilson1971primer}.

\begin{theo}[Hardy-Weinberg law in selectively neutral evolution]
	\label{theo0}
	Assume Assumptions  {\rm\ref{ass0}, \ref{ass2}}
		and {\rm\ref{ass3}} are fulfilled, and that the recruitment and mortality rates are the same for all genotypes.
	Then for any nonzero initial condition, the solutions of \eqref{SM0}, resp.\ \eqref{LM0}, verify
	\begin{equation}
	\label{eq22}
	\forall t\geq 0,\
	\frac{u_j\t x(t)}{\bfo\t x(t)} = \frac{u_j\t x(0)}{\bfo\t x(0)}:= p_j,\qquad j=A,a
	\end{equation}
	where $p_A+p_a=1$, and
	\begin{equation}
	\label{eq21}
	\lim_{t\to +\infty} x(t) =
	c_{\rm sn}^*\left(
	\begin{pmatrix}
	p_A^2\\ 2p_Ap_a \\ p_a^2
	\end{pmatrix}
	\right)
	 \begin{pmatrix}
	p_A^2\\ 2p_Ap_a \\ p_a^2
	\end{pmatrix}
	\end{equation}
	for $c_{\rm sn}^*$ defined in Lemma {\rm \ref{le44}}.
\end{theo}

Theorem \rm\ref{theo0} establishes that in absence of asymmetric competition among genotypes, systems \eqref{SM0} and \eqref{LM0} fulfil the Hardy-Weinberg Principle: the allele frequencies remain constant over time and determine the asymptotic genotype frequencies; while the total number of individuals converges to the unique population level (defined by \eqref{eq19}) which forms the carrying capacity for this relative repartition.

\begin{proof}[Proof of Theorem \rm\ref{theo0}]
\mbox{}

\noindent $\bullet$
	For any $j=A,a$, one has
	\begin{equation}
	u_j\t \dot x = u_j\t \left(
	\vphantom{\hat A}
	m_{\rm sn}(b_{\rm sn}(v\t x)) \Aset(x) -\mu_{\rm sn}(w\t x) x
	\right)
	= \left(
	\vphantom{\hat A}
	m_{\rm sn}(b_{\rm sn}(v\t x)) -\mu_{\rm sn}(w\t x)
	\right) u_j\t x
	\end{equation}
	using property {\it\ref{ppi}} in Lemma \ref{ppA}.
	Therefore
	\begin{equation}
	\label{eq24}
	\bfo\t \dot x
	= (u_A+u_a)\t \dot x 
	= (m_{\rm sn}(b_{\rm sn}(v\t x)) -\mu_{\rm sn}(w\t x)) \bfo\t x
	\end{equation}
	The two ratios $\displaystyle\frac{u_j\t x(t)}{\bfo\t x(t)}$ are defined for any $t\geq 0$ and differentiable with respect to time, with
	\begin{equation}
	\label{HW}
	\frac{d}{dt} \left(
	\frac{u_j\t x}{\bfo\t x}
	\right)
	= \frac{(u_j\t \dot x)(\bfo\t x)-(u_j\t x)(\bfo\t \dot x)}{(\bfo\t x)^2}
	= (m_{\rm sn}(b_{\rm sn}(v\t x)) -\mu_{\rm sn}(w\t x))
	\frac{(u_j\t x)(\bfo\t x)-(u_j\t x)(\bfo\t x)}{(\bfo\t x)^2}
	= 0
	\end{equation}
	which yields \eqref{eq22} by integration.
	
\noindent $\bullet$
One now studies the evolution of the components of the vector $x$.
One may rewrite  formally \eqref{eqsf} as
\begin{equation}
\label{eqsf2}
\dot x =  \hat m(t) \Aset(x) -\hat\mu(t) x
\end{equation}
where for simplicity we put the scalar functions $\hat m(t) :=m_{\rm sn}(b_{\rm sn}(v\t x(t)))$ and $\hat\mu(t) := \mu_{\rm sn}(w\t x(t))$.
Using \eqref{eq22}, which has just been demonstrated, one may write, {\em for any $t\geq 0$},
\begin{equation}
\label{eqsf3}
\Aset (x(t)) = \frac{1}{\bfo\t x(t)}
\begin{pmatrix}
(u_A\t x(t))^2 \\ 2(u_A\t x(t))(u_a\t x(t)) \\ (u_a\t x(t))^2
\end{pmatrix}
= (\bfo\t x(t)) \begin{pmatrix}
p_A^2 \\ 2p_Ap_a \\ p_a^2
\end{pmatrix}\ .
\end{equation}
Consequently, one obtains from \eqref{eqsf2} the identities:
\begin{equation}
\label{eqsf4}
\forall t\geq 0,\
\hat m(t) (\bfo\t x(t))
= \frac{1}{p_A^2} \left(
\vphantom{\int}
\dot x_1 + \hat\mu(t) x_1(t)
\right)
= \frac{1}{2p_Ap_a} \left(
\vphantom{\int}
\dot x_2 + \hat\mu(t) x_2(t)
\right)
= \frac{1}{p_a^2} \left(
\vphantom{\int}
\dot x_3 + \hat\mu(t) x_3(t)
\right)
\end{equation}
One then deduces that
\begin{equation}
\label{eqsf5}
\frac{d}{dt}\left(
\frac{x_1}{p_A^2}-\frac{x_2}{2p_Ap_a}
\right) + \hat\mu(t) \left(
\frac{x_1}{p_A^2}-\frac{x_2}{2p_Ap_a}
\right)
= \frac{d}{dt}\left(
\frac{x_3}{p_a^2}-\frac{x_2}{2p_Ap_a}
\right) + \hat\mu(t) \left(
\frac{x_3}{p_a^2}-\frac{x_2}{2p_Ap_a}
\right)
= 0\ .
\end{equation}
The function $\hat\mu(t)$ is uniformly bounded from below by a positive constant on the set $\{ x\in\Rset_+^3\ :\ \bfo\t x \leq c^*\}$, due to Theorem \ref{le01}.
Thus $\int_0^{+\infty} \hat\mu(t)\ dt =+\infty$, and one deduces by integration of \eqref{eqsf5} that the limits in the following formula exist, and consequently that the identities themselves are true:
\begin{equation}
\label{eq777}
\lim_{t\to +\infty} \left(
\frac{x_1(t)}{p_A^2}-\frac{x_2(t)}{2p_Ap_a}
\right)
= \lim_{t\to +\infty} \left(
\frac{x_3(t)}{p_a^2}-\frac{x_2(t)}{2p_Ap_a}
\right)
= 0\ .
\end{equation}
Therefore, the evolution occurs asymptotically on the half-line
\begin{equation}
\left\{
x\in\Rset_+^3\ :\ \exists c>0, x=c \begin{pmatrix}
p_A^2 \\ 2p_Ap_a \\ p_a^2
\end{pmatrix}
\right\}\ .
\end{equation}

\noindent $\bullet$
Now, the evolution of the state $x(t) = c(t) \begin{pmatrix}
p_A^2 \\ 2p_Ap_a \\ p_a^2
\end{pmatrix}$ on the previous half-line is dictated by the evolution of $c(t)$.
The fact that $p_A+p_a=1$ implies that $\bfo\t x(t) = c(t)$, and the function $c$ fulfils the scalar differential equation
\begin{equation}
\dot c = (\hat m(t) - \hat\mu(t))c(t)
= \left(
m_{\rm sn}\circ b_{\rm sn}\left(
c(t)v\t \begin{pmatrix}
p_A^2 \\ 2p_Ap_a \\ p_a^2
\end{pmatrix}
\right)- \mu_{\rm sn}\left(
c(t)w\t \begin{pmatrix}
p_A^2 \\ 2p_Ap_a \\ p_a^2
\end{pmatrix}
\right)\right) c(t)
\end{equation}
The latter possesses two points of equilibrium, namely 0 and $c_{\rm sn}^*\left(\begin{pmatrix}
p_A^2 \\ 2p_Ap_a \\ p_a^2
\end{pmatrix}\right)$, by definition of the map $c_{\rm sn}^*$ (see Lemma \ref{le44}).
	Due to Assumption \ref{ass2}, one has $m_{\rm sn} (0) > \mu_{\rm sn}(0)$, so the first equilibrium is unstable, while the second one is globally asymptotically stable.
	This achieves the proof of Theorem \ref{theo0}.
\end{proof}

\comment{
From the two proposed inheritance models and using the previous result one can deduce the Verhulst logistic equation \cite{verhulst1838notice} widely used in population dynamics.

Let $ z \in \Rset_{+}  $ be, such that $  z:=\bfo\t x $. Using equation \eqref{eqsf} and defined $v_i:=v_{\rm sn}$, $w_i:=w_{\rm sn}$ for $i = 1,2,3$, one can rewrite equations \eqref{SM0} and \eqref{LM0} as 
\begin{equation}
	\label{z}
	\dot z=z (m_{{\rm sf}}(a^{*}(z)) -\mu_{{\rm sf}}(z))
\end{equation} 
In particular, if one rewrite $ m_{{\rm sf}}(a^{*}(z)):=m_0-m z $ and  $ \mu_{U {\rm sf}}(z):=\mu_0+\mu z $, from \eqref{z} one has 
\begin{equation}
	\label{z0}
	\dot z=z ((m_0-m z)-(\mu_0+\mu z))
\end{equation} 
Therefore, for  $ \dot z=0 $ and $ z^*>0 $ it is satisfied
\begin{equation}
	\label{z*}
	m_0-c_m z^{*}=\mu_0+c_\mu z^{*},
\end{equation} in other words 
\begin{equation}
	z^{*}=\frac{\mu_0-m_0}{\mu+m}.
\end{equation} Defining $ r:=\mu_0-m_0 $, from equation \eqref{z0} one can deduce successively  
\begin{align}
	\label{z1}
	\dot z & = z ((m_0-\mu_0)-z(m +\mu))\\
	\dot z & = z (z^{*}(m +\mu)-z(m +\mu))\\
	\dot z & = (m_0-\mu_0) z (z^{*}-z)(m +\mu)(m_0-\mu_0)^{-1} \\
	\dot z & = (m_0-\mu_0) z (z^{*}-z)\frac{1}{z^{*}} \\
	\dot z & = r z \left(1-\frac{z}{z^{*}}\right), 
\end{align} this is the Verhulst logistic equation.
}

\section{Analysis of the selection case}
\label{se53}

We now state the result that describes the asymptotic behaviour in the {\em dominant} and {\em codominant} cases.

\begin{theo}[Asymptotic convergence to the homozygous equilibrium with higher fitness]
	\label{theoI}
	Assume Assumptions {\rm\ref{ass0}}, {\rm\ref{ass1}}, {\rm\ref{ass2}} and {\rm\ref{ass3}} are fulfilled.
	Then, for any nonzero initial condition, the solution of \eqref{SM0}, resp.\ \eqref{LM0}, satisfies 
	\begin{equation}
	\lim_{t\to +\infty} x(t)=c_1^*e_1
	\end{equation}
	if allele $A$ is initially present, and otherwise
	\begin{equation}
	\lim_{t\to +\infty} x(t)= c_3^* e_3\ .
	\end{equation} 
\end{theo}

Theorem \ref{theoI} states that, provided that the allele $ A $ is initially present, the system converges asymptotically towards the homozygous equilibrium of higher fitness.
When only the allele $ a $ is present, it goes towards the other homozygous nonzero equilibrium, or towards extinction if the latter does not exist.
In both cases, the asymptotic population levels $c_j^*$, $j=1,3$, correspond to the monomorphic equilibria defined in Lemma \ref{monoEqui}.\\

The case of monomorphic trajectories has been already studied in Lemma \ref{le1}.
The proof of Theorem \ref{theoI} in the general case of polymorphic trajectories is conducted in the remaining part of the present section, based on careful study of the evolution of each allele $j=A,a$ in the population, and then of the evolution of each genotypic population $i=1,2,3$.
Central use will be made of the notions of mean allelic rates and of the formulas \eqref{ady}, introduced in Section \ref{se518}.
Most of the demonstration steps below are similar for \eqref{SM0} and \eqref{LM0}, and will be treated altogether for both systems.
For simplicity we drop, whenever possible, the indices referring to the system considered, and simply write the allelic evolution
\begin{equation}
\label{eq662}
u_j\t \dot x
= \left(
\tm_j(x(t)) - \tmu_j(x(t))
\right) u_j\t x(t),\qquad j=a,A\ .
\end{equation}

We first demonstrate in the following result that the ratio of the allelic frequencies evolves in a strictly monotone way.

\begin{lemm}
	\label{le2}
	For any polymorphic trajectory,
	\begin{equation}
	\label{eq6}
	\forall t\geq 0,\qquad \frac{d}{dt} \left(
	\frac{u_a\t x(t)}{u_A\t x(t)}
	\right) = \left(
	\vphantom{\vec{A}}
	\tm_a(x(t)) - \tmu_a(x(t))-\tm_A(x(t))+\ \tmu_A(x(t))
	\right) \left(
	\frac{u_a\t x(t)}{u_A\t x(t)}
	\right) \leq 0\ ,
	\end{equation}
and the previous inequality is {\em strict} for any $t>0$.
\end{lemm}
\begin{proof}
\mbox{}

\noindent $\bullet$
As a consequence of Lemma \ref{le15}, the ratio $\displaystyle\frac{u_a\t x(t)}{u_A\t x(t)}$ is defined for any $t\geq 0$ along a polymorphic trajectory, and it is differentiable with respect to time.
	One has therefore
	\begin{equation}
	\forall t\geq 0,\qquad \frac{d}{dt} \left(
	\frac{u_a\t x(t)}{u_A\t x(t)}
	\right)
	= \left(\frac{u_a\t \dot x}{u_a\t x(t)}-\frac{u_A\t \dot x}{u_A\t x(t)} \right) \frac{u_a\t x(t)}{u_A\t x(t)}
	\end{equation}
which, thanks to \eqref{eq662}, gives the equality part of \eqref{eq6}.
Formulas \eqref{eq102} and \eqref{eq10} provide the non-strict version of \eqref{eq6}.

\noindent $\bullet$
To prove the {\em strict} inequality, consider first system \eqref{SM0}.
One has
\begin{multline}
\label{eq131}
\left(
\vphantom{\hat A}
\tm_{\ref{SM0},A}(x)-\tmu_{\ref{SM0},A}(x) -\tm_{\ref{SM0},a}(x) +\tmu_{\ref{SM0},a}(x)
\right)\\
= \left(
\vphantom{\hat A}
\tm_{\ref{SM0},A}(x)-m_2(b^*(x))-\tmu_{\ref{SM0},A}(x) + \mu_2(w\t x)
\right)
+
\left(
\vphantom{\hat A}
m_2(b^*(x))-\tm_{\ref{SM0},a}(x)+\tmu_{\ref{SM0},a}(x) - \mu_2(w\t x)
\right)
\end{multline} 
and due to Lemma \ref{oma}, each of the two expressions between parentheses in the right-hand side of \eqref{eq131} is {\em nonnegative}.

Moreover, at any polymorphic point, one has (see the definitions in \eqref{eqmA5}):
	\begin{subequations}
		\label{eq57}
		\begin{gather}
		\label{eq57a}
		\hspace{-.3cm}
		\tm_{\ref{SM0},A}(x)-m_2(b^*(x))-\tmu_{\ref{SM0},A}(x) + \mu_2(w\t x)
		= \frac{x_1}{u_A\t x}
		\left(
		\vphantom{\hat A}
		m_1(b^*(x))-m_2(b^*(x))
		+ \mu_2(w\t x)  - \mu_1(w\t x)
		\right)\\
		\label{eq57b}
		\hspace{-.3cm}
		m_2(b^*(x))-\tm_{\ref{SM0},a}(x)+\tmu_{\ref{SM0},a}(x) - \mu_2(w\t x)
		= \frac{x_3}{u_a\t x}
		\left(
		\vphantom{\hat A}
		m_2(b^*(x))-m_3(b^*(x))
		+ \mu_3(w\t x)  - \mu_2(w\t x)
		\right)
		\end{gather}
	\end{subequations}
A fundamental point now is that, due to the fact that $x_1\neq 0$ and $x_3\neq 0$ at any polymorphic point, one deduces from the second part of Assumption \ref{ass1} that {\em at least one of the two nonnegative expressions $m_1(b^*(x))-m_2(b^*(x)) + \mu_2(w\t x)  - \mu_1(w\t x)$ and $m_2(b^*(x))-m_3(b^*(x)) + \mu_3(w\t x)  - \mu_2(w\t x)$ is positive}.

As all genotypes are present along a polymorphic trajectory when $t>0$ (see Lemma \ref{le15}), one gets that, along any polymorphic trajectory, at least one of the two nonnegative expressions
\[
\frac{x_1(t)}{u_A\t x(t)}
		\left(
		\vphantom{\hat A}
		m_1(b^*(x(t)))-m_2(b^*(x(t)))
		+ \mu_2(w\t x(t))  - \mu_1(w\t x(t))
		\right)
\]
and
\[
		\frac{x_3(t)}{u_a\t x(t)}
		\left(
		\vphantom{\hat A}
		m_2(b^*(x(t)))-m_3(b^*(x(t)))
		+ \mu_3(w\t x(t))  - \mu_2(w\t x(t))
		\right)
\]
 is indeed {\em positive} whenever $t>0$.
 This in turn shows that along any polymorphic trajectory of equation \eqref{SM0}, $\left(
\vphantom{\hat A}
\tm_{\ref{SM0},a}(x(t))-\tmu_{\ref{SM0},a}(x(t)) -\tm_{\ref{SM0},A}(x(t)) +\tmu_{\ref{SM0},A}(x(t))
\right) <0$ for any $t> 0$, and thus the strict inequality in \eqref{eq6}.
This achieves the proof of Lemma \ref{le2} in the case of equation \eqref{SM0}.
 
 \noindent $\bullet$
The same argument holds for system \eqref{LM0}.
 The counterpart of the formulas \eqref{eq57} is obtained by noticing that, at any polymorphic point, one has (see \eqref{eqmA8}):
\begin{subequations}
\label{eq901}
		\begin{eqnarray}
		\label{eq1L}
		\nonumber
		\lefteqn{\tm_{\ref{LM0},A}(x) - m_2(b^*(\Aset(x))) - \tmu_{\ref{LM0},A}(x) + \mu_2(w\t x)}\\
		& = &
		\nonumber
		\frac{\alpha_1(x)}{\alpha_1(x)+\frac{1}{2}\alpha_2(x)}
		\left(
		\vphantom{\hat A}
		m_1(b^*(\Aset(x)))-m_2(b^*(\Aset(x)))
		\right) + \frac{x_1}{u_A\t x} ( \mu_2(w\t x)  - \mu_1(w\t x))\\
		& = &
		\label{eq90L}
		\frac{u_A\t x}{\bfo\t x}
		\left(
		\vphantom{\hat A}
		m_1(b^*(\Aset(x)))-m_2(b^*(\Aset(x)))
		\right) + \frac{x_1}{u_A\t x} ( \mu_2(w\t x)  - \mu_1(w\t x))
		\end{eqnarray}
		and similarly
		\begin{multline}
		m_2(b^*(\Aset(x)))-\tm_{\ref{LM0},a}(x)+\tmu_{\ref{LM0},a}(x) - \mu_2(w\t x)\\
		= 			\frac{u_a\t x}{\bfo\t x}
		\label{eq91L}
		\left(
		\vphantom{\hat A}
		m_2(b^*(\Aset(x)))-m_3(b^*(\Aset(x)))
		\right)
		+ \frac{x_3}{u_a\t x} (\mu_3(w\t x)  - \mu_2(w\t x))\ .
		\end{multline}
\end{subequations}
 Using the adequate version of the second part of Assumption \ref{ass1} allows to obtain in the same manner than before the strict inequality in \eqref{eq6} for equation \eqref{LM0}, and finally achieves the proof of Lemma \ref{le2}.
\end{proof}

We gather in the following result a series of estimates related to the genotypic frequencies.

\begin{lemm}
	\label{le3}
	For any polymorphic trajectory,
	there exists $c_1> 0$ such that
	\begin{equation}
	\label{eq14}
	\forall t\geq 0,\qquad u_a\t x(t) \leq c_1 (u_A\t x(t))\ .
	\end{equation}
	Moreover,
	 for any $\varepsilon> 0$, there exist $c_2,c_3\geq 0$ such that $\forall t\geq \varepsilon$
	\begin{gather}
	\label{x2c2x1}
	x_2(t) \leq c_2 x_1(t),\quad x_3(t) \leq c_3 x_2(t),\quad
	\frac{x_1(t)}{u_A\t x(t)}
	\geq \frac{1}{1+\frac{1}{2}c_2},\quad
	\frac{x_3(t)}{u_a\t x(t)}
	\leq \frac{c_3}{c_3+\frac{1}{2}}
	\end{gather}

\end{lemm}

\begin{proof}
\mbox{}

\noindent $\bullet$
	Formula \eqref{eq14} comes as direct consequence of Lemma \ref{le2}.
	As a matter of fact, integrating \eqref{eq6} gives
	\begin{equation}
	\label{eq11}
	\forall t\geq 0,\qquad
	\left(
	\frac{u_a\t x(t)}{u_A\t x(t)}
	\right)
	= e^{\int_0^t \left( \tm_a(x(t)) - \tmu_A(x(t))-\tm_A(x(t))+\ \tmu_A(x(t))\right) dt} \left(
	\frac{u_a\t x(0)}{u_A\t x(0)}
	\right)
	\leq \left(
	\frac{u_a\t x(0)}{u_A\t x(0)}
	\right)
	\end{equation} and therefore
	
	\begin{equation}
	\forall t\geq 0,\qquad u_a\t x(t) \leq \frac{u_a\t x(0)}{u_A\t x(0)}
	(u_A\t x(t))
	:= c_1 (u_A\t x(t))\ .
	\end{equation}

\noindent $\bullet$
Consider first system \eqref{SM0}.
	For any polymorphic trajectory one has $x_1(\varepsilon)> 0$ for any $\varepsilon> 0$, see the proof of Lemma \ref{le15}.
	Within the present demonstration, one assumes for simplicity that $\varepsilon$ may be taken as zero, i.e.\ $x_1(0)> 0$.
	Choose then positive constants $c_2, c_3$ such that
	\begin{equation}
	\label{eq8}
	c_2 > \max\left\{
	\frac{x_2(0)}{x_1(0)}; 2c_1
	\right\},\qquad
	c_3 < \max\left\{
	\frac{x_2(0)}{x_3(0)};\frac{2}{c_1}
	\right\}
	\end{equation}
	where $c_1$ is the positive constant in \eqref{eq14}.
	One deduces successively from Assumption \ref{ass1} and \eqref{eq14} that
	\begin{eqnarray}
	\label{c1M}
	u_a\t M(x)x
	& = & 
	\nonumber
	\frac{1}{2} m_2(b^*(x))x_2+ m_3(b^*(x))x_3\\
	& \leq &
	\nonumber 
	\frac{1}{2} m_2(b^*(x))x_2+ m_2(b^*(x))x_3\\
	& = &
	\nonumber 
	m_2(b^*(x))(u_a\t x)\\
	& \leq &
	\nonumber 
	c_1 m_2(b^*(x))(u_A\t x)\\
	& = &
	\nonumber
	c_1 m_2(b^*(x))\left(
	x_1+\frac{1}{2} x_2
	\right)\\
	& \leq &
	\nonumber
	c_1\left(
	m_1(b^*(x))x_1+\frac{1}{2} m_2(b^*(x))x_2
	\right)\\
	& = &c_1(u_A\t M(x)x)
	\end{eqnarray}
	Using by turns \eqref{Sb}, \eqref{c1M}, \eqref{eq8} and \eqref{Sa}, one deduces
	\begin{eqnarray}
	\dot{x}_2
	& = &
	\nonumber
	\frac{2}{\bfo\t M(x)x} (u_A\t M(x)x)(u_a\t M(x)x) - \mu_2(w\t x) x_2\\
	& \leq &
	\nonumber
	2c_1\frac{1}{\bfo\t M(x)x} (u_A\t M(x)x)^2 - \mu_2(w\t x) x_2\\
	& \leq &
	\nonumber
	c_2 \frac{1}{\bfo\t M(x)x} (u_A\t M(x)x)^2 - \mu_2(w\t x) x_2\\
	& = &
	\label{eq772}
	c_2  \left(
	\dot x_1+\mu_1(w\t x) x_1
	\right) - \mu_2(w\t x) x_2
	\end{eqnarray}
	Therefore,
	\begin{eqnarray}
	\dot x_2 -  c_2 \dot x_1
	& \leq &
	\nonumber
	c_2 \mu_1(w\t x) x_1 - \mu_2(w\t x) x_2\\
	& = &
	\nonumber
	- \mu_1(w\t x) \left(
	x_2-c_2x_1
	\right) +\left(
	\mu_1(w\t x) - \mu_2(w\t x)
	\right)x_2\\
	& \leq &
	\label{eq20}
	- \mu_1(w\t x) \left(
	x_2-c_2x_1
	\right)
	\end{eqnarray}
	where the last inequality has been deduced from the fact that $  \mu_2 \geq \mu_1  $, see Assumption \ref{ass1}.
	Integrating inequality \eqref{eq20} and using \eqref{eq8} yields for any $t\geq 0$:
	\begin{equation}
	\forall t\geq 0,\qquad
	x_2(t)-c_2x_1(t)
	\leq e^{-\int_0^t \mu_1(w\t x(s)) ds} \left(
	x_2(0)-c_2x_1(0)
	\right)
	\leq 0
	\end{equation}
	and the first part of \eqref{x2c2x1}.
	
	On the other hand, one also deduces from \eqref{Sb}, \eqref{c1M}, \eqref{Sc} and the fact that $\mu_3\geq \mu_2$,
\begin{eqnarray}
		\dot{x}_2
		& = &
		\nonumber
		\frac{2}{\bfo\t M(x)x} (u_A\t M(x)x)(u_a\t M(x)x) - \mu_2(w\t x) x_2\\
		& \geq &
		\nonumber
		\frac{2}{c_1}\frac{1}{\bfo\t M(x)x} (u_a\t M(x)x)^2 - \mu_2(w\t x) x_2\\
		& \geq &
		\nonumber
		c_3\frac{1}{\bfo\t M(x)x} (u_a\t M(x)x)^2 - \mu_2(w\t x) x_2\\
		& = &
		\nonumber
		c_3(\dot x_3+ \mu_3(w\t x) x_3) - \mu_2(w\t x) x_2\\
		& \geq &
		\nonumber
		c_3\dot x_3+ \mu_3(w\t x) \left(
		c_3 x_3- x_2
		\right)
\end{eqnarray}
	Therefore
	\begin{equation}
	\label{eq775}
	\dot x_2 - c_3\dot x_3 \geq - \mu_3(w\t x) \left(
	x_2 - c_3 x_3
	\right)
	\end{equation}
	which, taking into account the fact that $c_3 < \displaystyle\frac{x_2(0)}{x_3(0)}$, yields
	\begin{equation}
	\forall t\geq 0,\qquad
	x_2(t) - c_3 x_3(t) \geq 0\ .
	\end{equation}
This proves the second inequality in \eqref{x2c2x1}.
	The last two inequalities come as a consequence, using the fact that
	\[
	\frac{x_1(t)}{u_A\t x(t)} = \frac{x_1(t)}{x_1(t)+ \frac{1}{2}x_2(t)},\qquad
	\frac{x_3(t)}{u_a\t x(t)} = \frac{x_3(t)}{x_3(t)+ \frac{1}{2}x_2(t)}\ .
	\]
This concludes the proof of Lemma \ref{le3} in the case of system \eqref{SM0}.

\noindent $\bullet$
The proof is conducted similarly for system \eqref{LM0}.
We just quote here step by step the differences in the computations.
The analogue of formulas \eqref{eq772} and  \eqref{eq775} is proved as follows.
	Using \eqref{Lb}, Assumption \ref{ass1}, \eqref{eq14} and \eqref{La}, one deduces that
	\begin{eqnarray}
	\dot{x}_2
	& = &
	\nonumber
	\frac{2}{\bfo\t x} (u_A\t x)(u_a\t x)m_2(b^*(\Aset(x))) - \mu_2(w\t x) x_2\\
	& \leq &
	\nonumber
	2c_1\frac{(u_A\t x)^2}{\bfo\t x} m_1(b^*(\Aset(x))) - \mu_2(w\t x) x_2\\
	& \leq &
	\nonumber
	c_2 \frac{(u_A\t x)^2}{\bfo\t x} m_1(b^*(\Aset(x))) - \mu_2(w\t x) x_2\\
	& = &
	c_2  \left(
	\dot x_1+\mu_1(w\t x) x_1
	\right) - \mu_2(w\t x) x_2
	\end{eqnarray}
and therefore
	\begin{eqnarray}
	\dot x_2 -  c_2 \dot x_1
	& \leq &
	\nonumber
	c_2 \mu_1(w\t x) x_1 - \mu_2(w\t x) x_2\\
	& = &
	\nonumber
	- \mu_1(w\t x) \left(
	x_2-c_2x_1
	\right) +\left(
	\mu_1(w\t x) - \mu_2(w\t x)
	\right)x_2\\
	& \leq &
	\label{eq20L}
	- \mu_1(w\t x) \left(
	x_2-c_2x_1
	\right)
	\end{eqnarray}
	
	On the other hand, the fact that  $m_3\leq m_2$ and $\mu_3\geq \mu_2$ yields
	\begin{eqnarray}
	\dot{x}_2
	& = &
	\nonumber
	\frac{2}{\bfo\t x} (u_A\t x)(u_a\t x)m_2(b^*(\Aset(x))) - \mu_2(w\t x) x_2\\
	& \geq &
	\nonumber
	\frac{2}{c_1}\frac{(u_a\t x)^2}{\bfo\t x}m_3(b^*(\Aset(x))) - \mu_2(w\t x) x_2\\
	& \geq &
	\nonumber
	\frac{1}{c_3}\frac{(u_a\t x)^2}{\bfo\t x} m_3(b^*(\Aset(x))) - \mu_2(w\t x) x_2\\
	& = &
	\nonumber
	\frac{1}{c_3}(\dot x_3+ \mu_3(w\t x) x_3) - \mu_2(w\t x) x_2\\
	& \geq &
	\nonumber
	\frac{1}{c_3}\dot x_3+ \mu_3(w\t x) \left(
	\frac{1}{c_3} x_3- x_2
	\right)
	\end{eqnarray}
	which is the counterpart of \eqref{eq775}.
	The other steps of the proof are similar to the case of system \eqref{SM0}, and not repeated for the sake of space.
	This concludes the proof of Lemma \ref{le3}.
\end{proof}

We are now in position to show that the ratio of the allelic frequencies not only decreases, but also vanishes asymptotically.

\begin{lemm}
	\label{le5}
	For any polymorphic trajectory,
	\begin{equation}
	\label{eq112}
	\lim_{t\to +\infty} \frac{u_a\t x(t)}{u_A\t x(t)} = 0\ .
	\end{equation}
\end{lemm}

\begin{proof}
We consider in the whole demonstration a given polymorphic trajectory --- either of system \eqref{SM0} or of system \eqref{LM0} according to the context.

The ratio of allelic frequencies being decreasing  (as a consequence of Lemma \ref{le2}), there exists a nonnegative scalar $\lambda$ such that
\begin{equation}
\label{eq110}
	\lim_{t\to +\infty} \frac{u_a\t x(t)}{u_A\t x(t)} = \lambda\ .
\end{equation}
Due to the fact that $u_a\t x+u_A\t x=\bfo\t x$, one may deduce from \eqref{eq110} that the allelic frequencies also converge, namely:
\begin{equation}
\label{eq111}
	\lim_{t\to +\infty} \frac{u_a\t x(t)}{\bfo\t x(t)} = \frac{\lambda}{1+\lambda},\qquad
	\lim_{t\to +\infty} \frac{u_A\t x(t)}{\bfo\t x(t)} = \frac{1}{1+\lambda}\ .
\end{equation}

We assume by contradiction that
\begin{equation}
\label{eq119}
\lambda >0 \ .
\end{equation}
Our aim is to show that \eqref{eq119} is wrong, i.e.\ that $\lambda=0$.\\

\noindent $\bullet$
By Theorem \ref{le01}, the trajectories are (uniformly) bounded.
Therefore, by compactness, Assumption \ref{ass1} guarantees the existence of a certain $\zeta>0$ such that
\begin{subequations}
\label{eq144}
\begin{equation}
\forall t>0,\qquad m_1(b^*(x(t)))-m_3(b^*(x(t)))+\mu_3(w\t x(t))- \mu_1(w\t x(t))>\zeta
\end{equation}
 for system \eqref{SM0}; and
\begin{equation}
\forall t>0,\qquad m_1(b^*(\Aset(x(t))))-m_3(b^*(\Aset(x(t))))+\mu_3(w\t x(t))- \mu_1(w\t x(t))>\zeta
\end{equation}
for system \eqref{LM0}.
\end{subequations}
For the considered trajectory and the corresponding value $\zeta$, let the set $X$ be defined as
\begin{subequations}
\label{eq114}
\begin{equation}
\label{eq114a}
X := \left\{
x\in\Rset_+^3 :\ m_1(b^*(x))-m_2(b^*(x)) + \mu_2(w\t x)  - \mu_1(w\t x) > \frac{\zeta}{2}
\right\}
\end{equation}
for the case of system \eqref{SM0}, and as
\begin{equation}
\label{eq114b}
X := \left\{
x\in\Rset_+^3 :\ m_1(b^*(\Aset(x)))-m_2(b^*(\Aset(x)))+\mu_2(w\t x)- \mu_1(w\t x) > \frac{\zeta}{2}
\right\}
\end{equation}
\end{subequations}
for system \eqref{LM0}.
Notice that, due to \eqref{eq144},
\begin{subequations}
\label{eq117}
\begin{equation}
\label{eq117a}
x(t) \not\in X\ \Rightarrow\
m_2(b^*(x(t)))-m_3(b^*(x(t))) + \mu_3(w\t x(t))  - \mu_2(w\t x(t)) > \frac{\zeta}{2}
\end{equation}
for \eqref{SM0}, and for \eqref{LM0}:
\begin{equation}
\label{eq117b}
x(t) \not\in X\ \Rightarrow\
m_2(b^*(\Aset(x(t))))-m_3(b^*(\Aset(x(t))))+\mu_3(w\t x(t))- \mu_2(w\t x(t)) > \frac{\zeta}{2}\ .
\end{equation}
\end{subequations}

Now, observe that the derivative of $\displaystyle\frac{u_a\t x}{u_A\t x}$ appears in \eqref{eq6} as a {\em locally Lipschitz function} of the state variable $x$; and that along any polymorphic trajectory, the latter 
is {\em uniformly bounded} with {\em uniformly bounded time derivative}.
From this, we may deduce that this derivative is {\em uniformly continuous} with respect to the time variable.
 As the convergence property \eqref{eq110} holds, Barbalat's lemma \cite{Barbalat:1959aa,Farkas:2016aa}, establishes that the derivative converges to zero when $t\to +\infty$.
Inserting in the expression of the latter (see \eqref{eq6}), the decomposition in two nonnegative terms obtained in \eqref{eq131} and \eqref{eq57}  for \eqref{SM0}, and in \eqref{eq901} for \eqref{LM0}, one concludes by use of the hypothesis \eqref{eq119}, that
\begin{subequations}
\label{eq113}
\begin{gather}
\label{eq113a}
\lim_{t\to +\infty}
\frac{x_1(t)}{u_A\t x(t)}
		\left(
		\vphantom{\hat A}
		m_1(b^*(x(t)))-m_2(b^*(x(t)))
		+ \mu_2(w\t x(t))  - \mu_1(w\t x(t))
		\right)
= 0\\
\label{eq113b}
\lim_{t\to +\infty}
		\frac{x_3(t)}{u_a\t x(t)}
		\left(
		\vphantom{\hat A}
		m_2(b^*(x(t)))-m_3(b^*(x(t)))
		+ \mu_3(w\t x(t))  - \mu_2(w\t x(t))
		\right)
= 0
\end{gather}
\end{subequations}

\noindent $\bullet$
Assume first that, for the considered polymorphic trajectory and the set $X$ defined in \eqref{eq114},
\begin{equation}
\label{eq115}
\meas\ \left\{
t > 0\ :\ x(t) \in X
\right\} = +\infty\ ,
\end{equation}
where $\meas$ denotes the Lebesgue  measure.
				
Consider primarily the case of system \eqref{SM0}.
One derives from \eqref{eq57}, \eqref{x2c2x1} and the definition of $X$ in \eqref{eq114}, that
\begin{eqnarray}
\nonumber
\lefteqn{\tm_{\ref{SM0},a}(x(t)) - \tmu_{\ref{SM0},a}(x(t))-\tm_{\ref{SM0},A}(x(t))+\ \tmu_{\ref{SM0},A}(x(t))}\\
& \leq &
\nonumber
- \frac{x_1(t)}{u_A\t x(t)}
\left(
\vphantom{\hat A}
m_1(b^*(x(t)))-m_2(b^*(x(t))) + \mu_2(w\t x(t))  - \mu_1(w\t x(t))
\right)\\
& \leq &
\label{eq116}
- \frac{1}{1+\frac{1}{2}c_2} \frac{\zeta}{2}\ \chi_{x^{-1}(X)} (t)
\end{eqnarray}
where by definition the characteristic function $\chi_{x^{-1}(X)}(t)$ is equal to $1$ if $x(t)\in X$, $0$ otherwise.
Integrating now this inequality as in \eqref{eq11}, one gets for any $t\geq 0$,
\begin{equation}
	\left(
	\frac{u_a\t x(t)}{u_A\t x(t)}
	\right)
\leq \left(
\frac{u_a\t x(0)}{u_A\t x(0)}
\right)
\exp\left(
-\frac{1}{1+\frac{1}{2}c_2} \frac{\zeta}{2}\ \meas\ \left\{
s\in (0,t)\ :\ x(s) \in X
\right\}
\right)
\end{equation}
Due to the hypothesis made in \eqref{eq115}, the previous expression converges towards $0$ when $t\to +\infty$.
Therefore $\lambda = \displaystyle\lim_{t\to +\infty} \frac{u_a\t x(t)}{u_A\t x(t)} = 0$, which contradicts \eqref{eq119}.
This latter formula is thus wrong, which establishes \eqref{eq112} by contradiction.

The case of \eqref{LM0} is quite similar.
Due to \eqref{eq14}, one has for any $t\geq 0$,
\begin{equation}
\frac{u_A\t x(t)}{\bfo\t x(t)}  \geq \frac{1}{1+c_1}\qquad\text{ and }\qquad
\frac{u_a\t x(t)}{\bfo\t x(t)}  \leq \frac{c_1}{1+c_1}\ .
\end{equation}
From \eqref{eq90L}, one derives here
\begin{eqnarray}
\nonumber
\lefteqn{\tm_{\ref{LM0},a}(x(t)) - \tmu_{\ref{LM0},a}(x(t))-\tm_{\ref{LM0},A}(x(t))+\ \tmu_{\ref{LM0},A}(x(t))}\\
& \leq &
\nonumber
-\frac{u_A\t x}{\bfo\t x}
\left(
\vphantom{\hat A}
m_1(b^*(\Aset(x)))-m_2(b^*(\Aset(x)))
\right)
- \frac{x_1}{u_A\t x} ( \mu_2(w\t x)  - \mu_1(w\t x))\\
& \leq &
\nonumber
- \frac{1}{1+c_1}
\left(
\vphantom{\hat A}
m_1(b^*(\Aset(x)))-m_2(b^*(\Aset(x)))
\right)
- \frac{1}{1+\frac{1}{2}c_2} ( \mu_2(w\t x)  - \mu_1(w\t x))\\
& \leq &
-\min\left\{
\frac{1}{1+c_1}; \frac{1}{1+\frac{1}{2}c_2}
\right\}
\frac{\zeta}{2}\ \chi_{x^{-1}(X)} (t)
\end{eqnarray}
which is analogous to \eqref{eq116}.
The demonstration is then conducted in the same way and yields similarly $\lambda=0$.
By contradiction with \eqref{eq119}, this shows identity \eqref{eq112}.

\noindent $\bullet$
We now treat the case where \eqref{eq115} does not hold, that is:
\begin{equation}
\label{eq120}
\meas\ \left\{
t > 0\ :\ x(t) \in X
\right\} < +\infty\ .
\end{equation}
It is not possible to use here the same argument than previously, because the quantity $\displaystyle\frac{x_3(t)}{u_a\t x(t)}$ is bounded {\em from above}, contrary to $\displaystyle\frac{x_1(t)}{u_A\t x(t)}$ which is bounded {\em from below}, see \eqref{x2c2x1}.
 The measure of the set $\left\{
t > 0\ :\ x(t) \not\in X
\right\}$ is now infinite, and as a consequence of \eqref{eq117}, one deduces that
\begin{subequations}
\label{eq118}
\begin{equation}
\label{eq118a}
\meas\left\{
t\geq 0\ :\ m_2(b^*(x(t)))-m_3(b^*(x(t))) + \mu_3(w\t x(t))  - \mu_2(w\t x(t)) > \frac{\zeta}{2}
\right\} = +\infty
\end{equation}
for \eqref{SM0}, and for \eqref{LM0}:
\begin{equation}
\label{eq118b}
\meas\left\{
t\geq 0\ :\ m_2(b^*(\Aset(x(t))))-m_3(b^*(\Aset(x(t))))+\mu_3(w\t x(t))- \mu_2(w\t x(t)) > \frac{\zeta}{2}
\right\} = +\infty\ .
\end{equation}
\end{subequations}

From \eqref{eq113b} and \eqref{eq118}, one obtains here that necessarily:
\begin{equation}
\label{eq121a}
\lim_{t\to +\infty} \frac{x_3(t)}{u_a\t x(t)} = 0\ ,
\end{equation}
and thus
\begin{equation}
\label{eq121b}
\lim_{t\to +\infty} x_3(t) = 0\ .
\end{equation}

In view of the equations \eqref{Sc} for system \eqref{SM0} and \eqref{Lc} for system \eqref{LM0}, one deduces by invoking newly Barbalat's lemma that
\begin{equation}
\label{eq122}
\lim_{t\to +\infty} \dot{x}_3(t) = 0\ .
\end{equation}

By considering again the right-hand side of \eqref{Sc}, resp.\ \eqref{Lc}, one then gets from \eqref{eq121b} and \eqref{eq122} that
 \begin{equation}
\label{eq123a}
\lim_{t\to +\infty} u_a\t M(x(t)) x(t) = 0,\qquad
\text{ resp.\ }
\lim_{t\to +\infty} u_a\t x(t) = 0\ ,
\end{equation}
and therefore
 \begin{equation}
\label{eq123b}
\lim_{t\to +\infty} x_2(t) = 0\ .
\end{equation}

But \eqref{eq121b} and \eqref{eq123b} together yield
 \begin{equation}
\label{eq130}
\lambda = \lim_{t\to +\infty} \frac{u_a\t x(t)}{\bfo\t x(t)} = 0\ ,
\end{equation}
which contradicts \eqref{eq119}.
The latter is therefore wrong.
This establishes \eqref{eq112} in the case where \eqref{eq120} holds, and finally concludes the proof of Lemma \ref{le5}.
\end{proof}
 
As a remark, notice that using the techniques in the proof of Lemma \ref{le5}, one may show that the convergence is {\em exponential} in \eqref{eq112} whenever the following stronger form of Assumption \ref{ass1} holds: for any $x\in\Rset_+^3$,
$m_1(b^*(x))-m_2(b^*(x))+\mu_2(w\t x)- \mu_1(w\t x)>0$ for system \eqref{SM0}, or $m_1(b^*(\Aset(x)))-m_2(b^*(\Aset(x)))+\mu_2(w\t x)- \mu_1(w\t x)>0$ for system \eqref{LM0}.
Indeed \eqref{eq115} always holds in such cases.
On the contrary, there is no indication that the same property holds when $m_2(b^*(x))-m_3(b^*(x))+\mu_3(w\t x)- \mu_2(w\t x)>0$ for system \eqref{SM0}, or $m_2(b^*(\Aset(x)))-m_3(b^*(\Aset(x)))+\mu_3(w\t x)- \mu_2(w\t x)>0$ for system \eqref{LM0}.

Lemma \ref{le5} is sufficient to assert the asymptotic of the genotypic relative frequencies in \eqref{SM0} and \eqref{LM0}, as shown now.

\begin{lemm}
	\label{le55}
	For any polymorphic trajectory,
	\begin{equation}
	\label{eq145}
	\lim_{t\to +\infty} \frac{x_1(t)}{\bfo\t x(t)} = 1,\qquad
	\lim_{t\to +\infty} \frac{x_2(t)}{\bfo\t x(t)} = \lim_{t\to +\infty} \frac{x_3(t)}{\bfo\t x(t)} = 0\ .
	\end{equation}
\end{lemm}
\begin{proof}
	Recall that by definition
	\begin{equation}
	\frac{u_a\t x(t)}{u_A\t x(t)}
	= \frac{\frac{1}{2}x_2(t)+x_3(t)}{x_1(t)+\frac{1}{2}x_2(t)}
	\geq \frac{\frac{1}{2}x_2(t)+x_3(t)}{x_1(t)+x_2(t)+x_3(t)}
	\end{equation}
	which is at least equal to both nonnegative expressions
	\begin{equation}
	\frac{1}{2}\frac{x_2(t)}{x_1(t)+x_2(t)+x_3(t)}
	\qquad\text{ and }\qquad
	\frac{x_3(t)}{x_1(t)+x_2(t)+x_3(t)}
	\end{equation}
	Lemma \ref{le5} implies that both these ratios converge towards 0 when $t\to +\infty$ and this permits to conclude the proof of Lemma \ref{le55}.
\end{proof}

Due to Lemma \ref{le55} and the fact that the trajectories are bounded (see Theorem \ref{le01}), $x_2$ and $x_3$ converge towards 0.
We are finally in position to establish the asymptotic behaviour of the population size for each genotype.

\begin{lemm}
	\label{le6}
	For any polymorphic trajectory,
	\begin{equation}
	\lim_{t\to +\infty} x_1(t) = c_1^*,\qquad \lim_{t\to +\infty} x_2(t) = \lim_{t\to +\infty} x_3(t) = 0\ .
	\end{equation}
\end{lemm}
\begin{proof}
 
Consider e.g.\ system \eqref{SM0}.
Equation \eqref{Sa} may be written as
\begin{equation}
\dot x_1
=  \frac{(u_A\t M(x)x)^2}{\bfo\t M(x)x} - \mu_1(w\t x)x_1
= \left(
\frac{(u_A\t M(x)x)^2}{(\bfo\t M(x)x)x_1}
 - \mu_1(w\t x)
\right) x_1\ .
\end{equation}
Due to \eqref{eq145},
\begin{equation}
\hspace{-.3cm}
\lim_{t\to +\infty} \eta(t) = 0, \quad
\eta(t) := \left(
\frac{(u_A\t M(x)x)^2}{(\bfo\t M(x)x)x_1}
 - \mu_1(w\t x)
- m_1(b^*(x_1(t)e_1)) + \mu_1(w_1x_1(t))
\right)\ .
\end{equation}
Thus, for any polymorphic trajectory of \eqref{SM0} and any $\bar \eta>0$, there exists $T_{\bar\eta}>0$ such that, for any $t\geq T_{\bar\eta}$,
\begin{equation}
\label{eq146}
\left(
\vphantom{\hat{A}}
m_1(b^*(x_1(t)e_1)) - \mu_1(w_1x_1(t)) -\bar\eta
\right) x_1
\leq \dot x_1
\leq \left(
\vphantom{\hat{A}}
m_1(b^*(x_1(t)e_1)) - \mu_1(w_1x_1(t)) +\bar\eta
\right) x_1\ .
\end{equation}

For $\bar\eta>0$ sufficiently small, let $c^\pm_{\bar\eta}$ be the unique positive scalars such that
\begin{equation}
m_1(b^*(c^\pm_{\bar\eta} e_1)) - \mu_1(w_1c^\pm_{\bar\eta}) \pm\bar\eta = 0\ .
\end{equation}
By definition of $c^*_1$, see Lemma \ref{monoEqui}, one has
\begin {equation}
\label{eq147}
c^-_{\bar\eta} < c^*_1 < c^+_{\bar\eta}, \qquad
\lim_{\bar\eta \to 0^+} c^\pm_{\bar\eta} = c^*_1 \ .
\end{equation}

By \eqref{eq146}, one deduces that, for any sufficiently small $\bar\eta>0$,
\begin{equation}
c^-_{\bar\eta} \leq \liminf_{t\to +\infty} x_1(t) \leq \limsup_{t\to +\infty} x_1(t) \leq c^+_{\bar\eta}\ ,
\end{equation}
and finally
\begin{equation}
\liminf_{t\to +\infty} x_1(t) = \limsup_{t\to +\infty} x_1(t) \leq c^*_1\ ,
\end{equation}
by doing $\bar\eta\to 0^+$ and using \eqref{eq147}.
This demonstrates Lemma \ref{le6} in the case of system \eqref{SM0}.
System \eqref{LM0} is treated analogously.
\end{proof}
With the proof of Lemma \ref{le6}, the proof of Theorem \ref{theoI} is now complete.

\section{Conclusions and future issues}
\label{se7}

We have proposed a Mendelian inheritance model that considers the complexity of the life history of insects and its selective pressures. The  latter describes in continuous time a population with two main life phases, governed by birth of the three genotypes of two alleles, density-dependent mortality rates and constant rate of passage to reproductive phase. The model is represented by a system of six scalar ordinary differential equations (one for each genotype in each life phase). 

This first model was simplified using slow manifold theory, to obtain two classes of Mendelian inheritance models of single phase of dimension 3. Both models were derived from assuming one of the phases slower than the other.
These models allow assumptions about the selective pressures faced by genotypes in the ecological niches of each life-phase, in terms of recruitment and mortality, e.g.\ due to the use of larvicides or adulticides.

We have proved that, under appropriate assumptions, the two proposed classes demonstrate the fundamental behaviours expected in population dynamics and population genetics.
In a selectively neutral scenario, 
the population converges asymptotically to a carrying capacity, while Hardy-Weinberg law is valid: the allele frequencies in the polymorphic population are constant, and determine the asymptotic value of the genotypic relative frequencies.
In presence of different selective pressures, adaptive evolution occurs, and in case of dominance or codominance, the population is asymptotically made of individuals of the homozygous genotype having the highest fitness, at the corresponding carrying capacity.

The two proposed models, together with their possible extensions, are sufficiently general to consider the study of several issues of importance related to the use of insecticides and other adaptive phenomena.
Based on direct modelling of the genotypic dynamics (which is considered a correct approach for problems related to the use of insecticides \cite{taylor1975insecticide}), they allow for immediate study of the effects of larvicides and adulticides on the genotypes, as well as the effects on the allelic frequencies.
Besides, simple extensions of the models to incorporate migration would allow to evaluate the consequences of the natural or artificial addition of susceptible genotypes in the evolution of insecticide resistance, following the proposals made in \cite{comins1977development} and \cite{taylor1979suppression}.
Also, one could explore the consequences of the use of a larvicide and adulticide in the context of an evolution-proof insecticide following the line developed in \cite{read2009make,koella2009,gourley2011slowing}.
In addition, the models could be applied to situations in which one wishes to estimate the time taken by a resistant population to revert to a susceptible one, as done in \cite{schechtman2015costly}.

Looking forward, as most traits of evolutionary or economic importance are determined by several genes, an adequate understanding of the evolution of such traits may require the study of multi-locus models \cite{burger2011some}.
 It turns out that the heredity function which models the births is constructed in such a way as to allow extensions to more than two alleles, multiple loci (by means of Kronecker product between inheritance matrices) or polyploid cases.
The modelling procedure presented in the present paper thus offers the ability to accommodate more complicated inheritance configurations, with two or more loci with autosomal inheritance, accounting for e.g.\ sequential or mixed use of insecticides \cite{curtis1985theoretical,mani1985evolution,levick2017two,south2018insecticide}, non-genetic inheritance, unified maternal and autosomal inheritance (as in {\em Wolbachia}) and autosomal resistance for biological control \cite{hoffmann2013facilitating}.
As a last remark, notice that models of species having more than two main phases may be considered too, for example in the case of a holometabolous insect for which the phases of embryo, larva and pupa are all quite fast compared to the adult phase.

Mathematical models played a decisive role in reconciling Mendelian genetics with Darwin's theory of adaptive evolution \cite{ewens2011changes}. In regard to inheritance models, the genetic control for insect pest raises new issues with potentially valuable applications \cite{alphey2014genetic,hoffmann2015wolbachia}. We think that the modelling strategy presented here will help in this task.

\subsection*{Acknowledgments}

P.E.P.E.\ acknowledges the support of CONACyT in the framework of FEEI-PROCIENCIA program (POS007, POSG17-53, PVCT15-273, PVCT17-156, PVCT18-53). C.E.S.\ and P.E.P.E.\ acknowledge the CONACyT-PRONII program.
All authors acknowledge the support of the STIC AmSud project MOSTICAW (2016--2017).

\appendix

\section{Proof of technical results}
\label{se10}

\subsection{Proof of Lemma \ref{a0}}
\label{app1}

The left-hand side of \eqref{a*} is an increasing function of $b$ and varies from $0$ to $+\infty$; while the right-hand side is null if $x=0_3$ and a decreasing function otherwise.
Therefore, by the Implicit function Theorem, there exists a unique solution $b^*(x)$ to the scalar equation
\[
b - \sum_{i=1}^{3} v_im_i(b)x_i = 0
\]
which is of class $C^p$ if every function $m_i$ is of class $C^p$,  $p\in\Nset$.
\hfill $\square$

\subsection{Proof of Lemma \ref{ppA}}
\label{app2}

For Property {\it\ref{ppi}}, one can see that,  for any $j=A,a$,
\begin{equation}
\label{/A1}
u_j\t \Aset(x) = 
\frac{1}{\bfo\t x} ((u_j\t x)^{2}+(u_A\t x)(u_a\t x))= 
\frac{1}{\bfo\t x}(u_A\t x+u_a\t x)u_j\t x = 
u_j\t x.
\end{equation}
as $u_A+u_a= \bfo$.
 Using again this identify, property {\it\ref{ppii}} is deduced from the previous one, as:
\begin{equation}
\bfo\t \Aset(x) = 
u_A\t \Aset(x)+u_a\t \Aset(x)=
u_A\t x+u_a\t x =
\bfo\t x.
\end{equation}

Property {\it\ref{ppiii}} is trivial, and expresses the homogeneity of the function $ \Aset $. 
To show Property {\it\ref{ppiv}},
notice that for any $(j,k)\in\{1,3\}\times\{A,a\}$, $u_k\t e_j = 1$  if $(j,k)=(1,A)$ or $(3,a)$, and 0 otherwise;
and on the other hand that $\bfo\t e_j=1$.
Therefore
\begin{equation}
\Aset(e_j)=\frac{1}{\bfo\t e_j} \begin{pmatrix}
(u_A\t e_j)^2\\ 2(u_A\t e_j)(u_a\t e_j) \\  (u_a\t e_j)^2\end{pmatrix} = e_j
\end{equation}
for any $j=1,3$.
\hfill $\square$

\subsection{Proof of Lemma \ref{le0}}
\label{app3}

$\bullet$
Let $ x \in \mathbb{R}_{+}^{3} \setminus \{0_3\}  $ and $ \lambda>0 $.
By definition (see the statement of Lemma \ref{a0}),
 
\begin{equation}
b^*(\lambda x)=\lambda \sum_{i=1}^{3} v_im_i(b^*(\lambda x))x_i.
\end{equation}
Therefore 
\begin{equation}
\lambda  \min_{i=1,2,3} \{v_i/w_i\} \min_{i=1,2,3} \{m_i(b^*(\lambda x))\}w\t x
\leq b^*(\lambda x)
\leq \lambda  \max_{i=1,2,3} \{v_i /w_i\}\max_{i=1,2,3} \{m_i(b^*(\lambda x))\}w\t x. 
\end{equation} which may be rewritten as an inequality on $\lambda (w\t x)$:
\begin{equation}
\frac{b^*(\lambda x)}{\max_i \{v_i/w_i\} \max_{i}\{ m_i(b^*(\lambda x))\}}
\leq  \lambda (w\t x)
\leq \frac{b^*(\lambda x)}{\min_i \{v_i/w_i\} \min_{i}\{ m_i(b^*(\lambda x))\}}  
\end{equation}
Therefore, when $ \lambda \rightarrow +\infty $,
one has $\displaystyle\frac{b^*(\lambda x)}{\min_i\{v_i/w_i\} \min_{i} \{m_i(b^*(\lambda x))\}} \rightarrow + \infty $, and the convergence is uniform with respect to $ x $ such that $ w\t x=1 $.
Consequently, $ b^*(\lambda x) \rightarrow + \infty $  uniformly in $ x $ such that $w\t x=1 $.

From this one deduces that  $\displaystyle \lim_{\lambda\to +\infty} m_i(b^*(\lambda x))< \lim_{\lambda\to +\infty} \mu_i(\lambda) $, because $ m_i  $ decreases, $ \mu_{i} $ increases and Assumption \ref{ass3} holds.

\noindent $\bullet$
Let us now show the second property on $b^*$.
By definition one has, for any $x\in\Rset_+^3\setminus\{0\}$,
\[
b^*(x)=\sum_{i=1}^{3} v_i m_i(b^*(x))x_i.
\]
Therefore, for any $\lambda, \lambda'\geq 0$,
\[
\lambda \sum_{i=1}^{3} v_i m_i(b^*(\lambda x)) x_i   - \lambda' \sum_{i=1}^{3} v_i m_i(b^*(\lambda' x)) x_i 
= b^*(\lambda x)-b^*(\lambda' x)
\]
Subtracting and adding the term $ \lambda \sum_{i=1}^{3} v_i m_i(b^*(\lambda' x)) x_i$, one gets
\[
\lambda \sum_{i=1}^{3} v_i \left(
\vphantom{\vec{A}}
m_i(b^*(\lambda x)) - m_i(b^*(\lambda' x))
\right) x_i
+ (\lambda- \lambda') \sum_{i=1}^{3} v_i m_i(b^*(\lambda' x)) x_i 
= b^*(\lambda x)-b^*(\lambda' x)
\]
that is
\[
(\lambda- \lambda') \sum_{i=1}^{3} v_i m_i(b^*(\lambda' x)) x_i 
= b^*(\lambda x)-b^*(\lambda' x) - \lambda \sum_{i=1}^{3} v_i \left(
\vphantom{\vec{A}}
m_i(b^*(\lambda x)) - m_i(b^*(\lambda' x))
\right) x_i
\]
Assume e.g.\ $b^*(\lambda x)>b^*(\lambda' x)$.
Due to Assumption \ref{ass0}, the functions $m_i$, $i=1,2,3$, decrease.
Therefore $m_i(b^*(\lambda x)) - m_i(b^*(\lambda' x))<0$, and we deduce from the previous identity that $\lambda> \lambda'$.

One shows similarly that $b^*(\lambda x)<b^*(\lambda' x)$ implies $\lambda<\lambda'$.
In conclusion, $\lambda<\lambda' \Leftrightarrow b^*(\lambda x)<b^*(\lambda' x)$, and this shows that $b^*$ is increasing.
\hfill $\square$

\medskip

\bibliographystyle{unsrt}
\bibliography{mybibfile}

\begin{thebibliography}{10}

\bibitem{caraballo2014emergency}
Hector Caraballo and Kevin King.
\newblock Emergency department management of mosquito-borne illness: malaria,
  dengue, and {West Nile} virus.
\newblock {\em Emergency medicine practice}, 16(5):1--23, 2014.

\bibitem{201827}
Adnan~I. Qureshi, editor.
\newblock {\em Zika virus disease: from origin to outbreak}, chapter 2,
  Mosquito-Borne Diseases.
\newblock Academic Press, 2018.

\bibitem{world2014global}
World~Health Organization et~al.
\newblock A global brief on vector-borne diseases.
\newblock Technical report, World Health Organization, 2014.

\bibitem{roehrig2009arboviruses}
John~T Roehrig and Robert~S Lanciotti.
\newblock Arboviruses.
\newblock In {\em Clinical Virology Manual, Fourth Edition}, pages 387--407.
  American Society of Microbiology, 2009.

\bibitem{bhatt2013global}
Samir Bhatt, Peter~W Gething, Oliver~J Brady, Jane~P Messina, Andrew~W Farlow,
  Catherine~L Moyes, John~M Drake, John~S Brownstein, Anne~G Hoen, Osman
  Sankoh, et~al.
\newblock The global distribution and burden of dengue.
\newblock {\em Nature}, 496(7446):504, 2013.

\bibitem{shepard2016global}
Donald~S Shepard, Eduardo~A Undurraga, Yara~A Halasa, and Jeffrey~D Stanaway.
\newblock The global economic burden of dengue: a systematic analysis.
\newblock {\em The Lancet infectious diseases}, 16(8):935--941, 2016.

\bibitem{sharma2017insect}
Smriti Sharma, Rubaljot Kooner, and Ramesh Arora.
\newblock Insect pests and crop losses.
\newblock In {\em Breeding Insect Resistant Crops for Sustainable Agriculture},
  pages 45--66. Springer, 2017.

\bibitem{oerke2006crop}
E-C Oerke.
\newblock Crop losses to pests.
\newblock {\em The Journal of Agricultural Science}, 144(1):31--43, 2006.

\bibitem{douglas2018strategies}
Angela~E Douglas.
\newblock Strategies for enhanced crop resistance to insect pests.
\newblock {\em Annual review of plant biology}, 69:637--660, 2018.

\bibitem{valle2015controle}
Denise Valle, Thiago~Affonso Belinato, and AJ~Martins.
\newblock {\em Dengue: teorias e pr{\'a}ticas}, chapter Controle qu{\'\i}mico
  de {Aedes} aegypti, resist{\^e}ncia a inseticidas e alternativas, pages
  93--126.
\newblock Rio de Janeiro: Fiocruz, 2015.

\bibitem{ShawCatterucci}
Catteruccia~F. Shaw~WR.
\newblock Vector biology meets disease control: {Using} basic research to fight
  vector-borne diseases.
\newblock {\em Nature microbiology}, 4:20--34, 2019.

\bibitem{schechtman2015costly}
Helio Schechtman and Max~O Souza.
\newblock Costly inheritance and the persistence of insecticide resistance in
  {Aedes} aegypti populations.
\newblock {\em PloS one}, 10(5):e0123961, 2015.

\bibitem{hemingway2000insecticide}
Janet Hemingway and Hilary Ranson.
\newblock Insecticide resistance in insect vectors of human disease.
\newblock {\em Annual review of entomology}, 45(1):371--391, 2000.

\bibitem{koella2009}
Jacob~C. Koella, Penelope~A. Lynch, Matthew~B. Thomas, and Andrew~F. Read.
\newblock Towards evolution-proof malaria control with insecticides.
\newblock {\em Evolutionary Applications}, 2(4):469--480, 2009.

\bibitem{vontas2012insecticide}
J~Vontas, E~Kioulos, N~Pavlidi, E~Morou, A~Della~Torre, and Hilary Ranson.
\newblock Insecticide resistance in the major dengue vectors aedes albopictus
  and aedes aegypti.
\newblock {\em Pesticide Biochemistry and Physiology}, 104(2):126--131, 2012.

\bibitem{Azambuja-Garcia:2017aa}
Gabriela de~Azambuja~Garcia.
\newblock {\em O {papel da resist\^encia a inseticidas e da densidade de Aedes
  aegypti na dissemina\c c\~ ao da Wolbachia em popula\c c\~ oes nativas do Rio
  de Janeiro, Brasil}}.
\newblock PhD thesis, Instituto {Oswaldo Cruz -- FIOCRUZ}, Rio de Janeiro,
  August 2017.

\bibitem{Garcia:2017aa}
Gabriela~A. Garcia, Gabriel Sylvestre, Mariana~R. David, Ademir Martins, Daniel
  A.~M. Villela, Fernando~B.S. Dias, Luciano~A. Moreira, and Rafael~Maciel
  de~Freitas.
\newblock The riddle solved on a local grocery store: the release of {Aedes}
  aegypti as resistant to pyrethroids as the wild population is essential for
  {Wolbachia} invasion.
\newblock In {\em Book of Abstracts of the 7th International Congress of the
  Society for Vector Ecology}, 2017.

\bibitem{daborn2004genetics}
Phillip~J Daborn, Gaelle Le~Goff, et~al.
\newblock The genetics and genomics of insecticide resistance.
\newblock {\em TRENDS in Genetics}, 20(3):163--170, 2004.

\bibitem{labbe2011evolution}
Pierrick Labb{\'e}, Haoues Alout, Luc Djogb{\'e}nou, Nicole Pasteur, and Mylene
  Weill.
\newblock Evolution of resistance to insecticide in disease vectors.
\newblock In {\em Genetics and evolution of infectious disease}, pages
  363--409. Elsevier, 2011.

\bibitem{hofbauer1998evolutionary}
Josef Hofbauer and Karl Sigmund.
\newblock {\em Evolutionary games and population dynamics}.
\newblock Cambridge university press, 1998.

\bibitem{ewens2011changes}
Warren~J Ewens.
\newblock What {Changes Has Mathematics Made to the Darwinian Theory?}
\newblock In {\em The Mathematics of Darwin's Legacy}, pages 7--26. Springer,
  2011.

\bibitem{schuster2011mathematical}
Peter Schuster.
\newblock Mathematical modeling of evolution. solved and open problems.
\newblock {\em Theory in Biosciences}, 130(1):71--89, 2011.

\bibitem{schuster2011mathematics}
Peter Schuster.
\newblock The mathematics of {Darwin's theory} of evolution: 1859 and 150 years
  later.
\newblock In {\em The Mathematics of Darwin's Legacy}, pages 27--66. Springer,
  2011.

\bibitem{burger2011some}
Reinhard B{\"u}rger.
\newblock Some mathematical models in evolutionary genetics.
\newblock In {\em The Mathematics of Darwin's Legacy}, pages 67--89. Springer,
  2011.

\bibitem{huillet:hal-00526859}
Thierry Huillet and Servet Martinez.
\newblock {Discrete evolutionary genetics. {Multiplicative} fitnesses and the
  mutation-fitness balance}.
\newblock {\em {Applied Mathematics}}, Vol 2(no 1):pp. 11--22, January 2011.
\newblock {\`a} paraitre dans: ''Applied Mathematics''.

\bibitem{schuster1983replicator}
Peter Schuster and Karl Sigmund.
\newblock Replicator dynamics.
\newblock {\em Journal of theoretical biology}, 100(3):533--538, 1983.

\bibitem{hofbauer1982game}
Josef Hofbauer, Peter Schuster, and Karl Sigmund.
\newblock Game dynamics in {Mendelian} populations.
\newblock {\em Biological Cybernetics}, 43(1):51--57, 1982.

\bibitem{barbosa2012importance}
Susana Barbosa and Ian~M Hastings.
\newblock The importance of modelling the spread of insecticide resistance in a
  heterogeneous environment: the example of adding synergists to bed nets.
\newblock {\em Malaria journal}, 11(1):258, 2012.

\bibitem{levick2017two}
Bethany Levick, Andy South, and Ian~M Hastings.
\newblock A two-locus model of the evolution of insecticide resistance to
  inform and optimise public health insecticide deployment strategies.
\newblock {\em PLoS computational biology}, 13(1):e1005327, 2017.

\bibitem{taylor1975insecticide}
C~Robert Taylor and JC~Headley.
\newblock Insecticide resistance and the evaluation of control strategies for
  an insect population.
\newblock {\em The Canadian Entomologist}, 107(3):237--242, 1975.

\bibitem{comins1977development}
Hugh~N Comins.
\newblock The development of insecticide resistance in the presence of
  migration.
\newblock {\em Journal of theoretical biology}, 64(1):177--197, 1977.

\bibitem{taylor1979suppression}
Charles~E Taylor and George~P Georghiou.
\newblock Suppression of insecticide resistance by alteration of gene dominance
  and migration.
\newblock {\em Journal of Economic Entomology}, 72(1):105--109, 1979.

\bibitem{crow1970introduction}
James~F Crow, Motoo Kimura, et~al.
\newblock An introduction to population genetics theory.
\newblock {\em An introduction to population genetics theory.}, 1970.

\bibitem{curtis1985theoretical}
CF~Curtis.
\newblock Theoretical models of the use of insecticide mixtures for the
  management of resistance.
\newblock {\em Bulletin of entomological research}, 75(2):259--266, 1985.

\bibitem{mani1985evolution}
GS~Mani.
\newblock Evolution of resistance in the presence of two insecticides.
\newblock {\em Genetics}, 109(4):761--783, 1985.

\bibitem{south2018insecticide}
Andy South and Ian~M Hastings.
\newblock Insecticide resistance evolution with mixtures and sequences: a
  model-based explanation.
\newblock {\em Malaria journal}, 17(1):80, 2018.

\bibitem{read2009make}
Andrew~F Read, Penelope~A Lynch, and Matthew~B Thomas.
\newblock How to make evolution-proof insecticides for malaria control.
\newblock {\em PLoS biology}, 7(4):e1000058, 2009.

\bibitem{gourley2011slowing}
Stephen~A Gourley, Rongsong Liu, and Jianhong Wu.
\newblock Slowing the evolution of insecticide resistance in mosquitoes: a
  mathematical model.
\newblock {\em Proceedings of the Royal Society A: Mathematical, Physical and
  Engineering Sciences}, 467(2132):2127--2148, 2011.

\bibitem{langemann2013multi}
Dirk Langemann, Otto Richter, and Antje Vollrath.
\newblock Multi-gene-loci inheritance in resistance modeling.
\newblock {\em Mathematical biosciences}, 242(1):17--24, 2013.

\bibitem{KLONOWSKI198373}
Wlodzimierz Klonowski.
\newblock Simplifying principles for chemical and enzyme reaction kinetics.
\newblock {\em Biophysical Chemistry}, 18(2):73 -- 87, 1983.

\bibitem{elbadry1966life}
EA~Elbadry and MSF Tawfik.
\newblock Life cycle of the mite {Adactylidium sp. (Acarina: Pyemotidae), a
  predator of thrips eggs in the United Arab Republic}.
\newblock {\em Annals of the Entomological Society of America}, 59(3):458--461,
  1966.

\bibitem{gould2010panda}
Stephen~Jay Gould.
\newblock {\em The panda's thumb: More reflections in natural history}.
\newblock WW Norton \& company, 2010.

\bibitem{welch1998shortest}
Craig~H Welch.
\newblock Shortest reproductive life.
\newblock {\em University of Florida Book of Insect Records}, 1998.

\bibitem{Robert-Jr:1991aa}
Robert~E O'Malley~Jr.
\newblock {\em Singular perturbation methods for ordinary differential
  equations}, volume~89.
\newblock Springer Science \& Business Media, 1991.

\bibitem{edward2012219}
A.W.F. Edwards.
\newblock Punnett's square.
\newblock {\em Studies in History and Philosophy of Science Part C: Studies in
  History and Philosophy of Biological and Biomedical Sciences}, 43(1):219 --
  224, 2012.
\newblock Data-Driven Research in the Biological and Biomedical Sciences On
  Nature and Normativity: Normativity, Teleology, and Mechanism in Biological
  Explanation.

\bibitem{wilson1971primer}
Edward~O Wilson and William~H Bossert.
\newblock {\em A primer of population biology}.
\newblock Sinauer Associates Sunderland, MA, 1971.

\bibitem{Barbalat:1959aa}
I.~Barb\u alat.
\newblock Syst\`emes d'\'equations diff\'erentielles d'oscillations non
  lin\'eaires.
\newblock {\em Rev. Math. Pures Appl.}, 4:267--270, 1959.

\bibitem{Farkas:2016aa}
B{\'a}lint Farkas and Sven-Ake Wegner.
\newblock Variations on {B}arb{\u{a}}lat's lemma.
\newblock {\em The American Mathematical Monthly}, 123(8):825--830, 2016.

\bibitem{hoffmann2013facilitating}
Ary~A Hoffmann and Michael Turelli.
\newblock Facilitating {Wolbachia} introductions into mosquito populations
  through insecticide-resistance selection.
\newblock {\em Proceedings of the Royal Society B: Biological Sciences},
  280(1760):20130371, 2013.

\bibitem{alphey2014genetic}
Luke Alphey.
\newblock Genetic control of mosquitoes.
\newblock {\em Annual review of entomology}, 59, 2014.

\bibitem{hoffmann2015wolbachia}
Ary~A Hoffmann, Perran~A Ross, and Gordana Ra{\v{s}}i{\'c}.
\newblock Wolbachia strains for disease control: ecological and evolutionary
  considerations.
\newblock {\em Evolutionary applications}, 8(8):751--768, 2015.

\end{thebibliography}

\end{document}